\documentclass[15pt]{amsart}
\usepackage{epsfig,amssymb,latexsym,amscd,stmaryrd}
\usepackage[utf8]{inputenc}
\usepackage{manfnt}
\usepackage{xcolor}
\usepackage[foot]{amsaddr}
\usepackage[all,2cell]{xy}
\xyoption{v2} \UseAllTwocells
\usepackage{hyperref}

\newtheorem{teo}{Theorem}[section]
\newtheorem{theo}[teo]{Theorem}
                \newtheorem{coro}[teo]{Corollary}
\newtheorem{lema}[teo]{Lemma}
\newtheorem{prop}[teo]{Proposition}

\newtheorem{defi}[teo]{Definition}
\newtheorem{obse}[teo]{Observation}

\newcommand{\modRG}{{}_{(R,G)}\,\mathcal M}

\newcommand{\modRH}{{}_{(R,H)}\,\mathcal M}

\newcommand{\modG}{{}_{G}\mathcal M}
\newcommand{\modH}{{}_{H}\mathcal M}

\newcommand{\id}{\ensuremath{\mathrm{id}}}

\begin{document}


\title[A short survey on observability]
{A short survey on observability}

\author{Walter Ferrer Santos} \address{Departamento de matem\'atica y
  aplicaciones, Cure \\Universidad de la Rep\'ublica\\ Tacuaremb\'o
  entre Av. Artigas y Aparicio Saravia\\CP 20000,
  Maldonado\\Uruguay\\wrferrer@cure.edu.uy} \thanks{The author would
  like to thank, Anii, Csic, Pedeciba, Prog. 720, Udelar, for
  partially financing activities related to the current research
  projects of the author.}
\begin{abstract}
   {\footnotesize The exploration of the notion of observability
     exhibits transparently the rich interplay between algebraic and
     geometric ideas in \emph{geometric invariant theory}. The concept
     of \emph{observable subgroup} was introduced in the early 1960s
     with the purpose of studying extensions of representations from
     an affine algebraic subgroup to the whole group. The extent of
     its importance in \emph{representation and invariant theory} in
     particular for Hilbert's $14^{\text{th}}$ problem was noticed
     almost immediately. An important strenghtening appeared in the
     mid 1970s when the concept of \emph{strong observability} was
     introduced and it was shown that the notion of observability can
     be understood as an intermediate step in the notion of
     reductivity (or semisimplicity), when adequately
     generalized. More recently starting in 2010, the concept of
     observable subgroup was expanded to include the concept of
     \emph{observable action} of an affine algebraic group on an
     affine variety, launching a series of new applications. In 2006
     the related concept of \emph{observable adjunction} was
     introduced, and its application to module categories over tensor
     categories was noticed. In the current survey, we follow
     (approximately) the historical development of the subject
     introducing along the way, the definitions and some of the main
     results including some of the proofs. For the unproven parts,
     precise references are mentioned.}
\end{abstract} 

\maketitle


\section{Introduction}
\label{section:intro}
The concept of observable subgroup of an affine algebraic group $G$
was introduced by A. Bialynicki--Birula, G. Hochschild and G.D. Mostow
in 1963 in \cite{kn:hmbb}: \emph{Extension of representations of algebraic linear groups} (hereafter referred to as ERA).

Initially the notion of observability was related to the following
situation.

Assume that $H \subseteq G$ is a pair of a subgroup and a group. We
say that a representation $(V,\rho)$ of $G$ is an extension of a
representation $(U,\sigma)$ of $H$ if: $U \subseteq V$ and the action
$\rho: G \times V \to V$ restricts to $\sigma:H \times U \to
U$\footnote{For the above question to make sense, the general
  definition has to be adapted to particular situations involving a
  basic field --where the representations are defined-- and a precise
  description of the actions we are working with (i.e. maps such as
  $\rho$ and $\sigma$ above) that have to be adapted to the additional
  structure of the groups under consideration -- analytic,
  differentiable, algebraic, etc.}.

The main question adressed by the authors of \cite{kn:hmbb} concerns
the following problem: in the case that $H \text{ and } G$ are affine
algebraic groups, and the representations are finite dimensional and
rational, does every representation of $H$ admits an extension? In
the situation that the answer is positive the group is said to be \emph{observable}.

In the introduction of ERA the authors write:

\smallskip
\begin{center}
\begin{minipage}{.8\textwidth}
{\small {\em Let $G$ be an algebraic linear group over an arbitrary field
  $F$. If $\rho$ is a rational representation of G by linear
  automorphisms of a finite--dimensional $F$--space $U$, we refer to
  this structure $(U, \rho)$ by saying that $U$ is a
  finite--dimensional rational $G$--module. A $G$--module that is a
  sum (not necessarily direct) of finite--dimensional rational
  $G$--modules is called a rational $G$--module. Let $H$ be an
  algebraic subgroup of $G$. We are interested in determining when
  every finite--dimensional rational representation of $H$ can be
  extended to a rational representation of $G$, i.e., when every
  finite--dimensional rational $H$--module can be imbedded as a
  $H$--submodule in a rational $G$--module.}}
\end{minipage}
\end{center}
\bigskip

{\sc Notations and prerequisites.} In this paper we assume that the
reader is familiar with the basic results and notations of the theory of affine algebraic groups its actions and representations which appear
--eventually with slight differences-- in the intial chapters of the standard textbooks on the subject such as:
A. Borel's \cite{kn:borelbook}, C. Chevalley's \cite{kn:Che}, G. Hochschild's \cite{kn:hobook}, J. E. Humphrey's \cite{kn:humph}, T.A. Springer's \cite{kn:sprbook} or the more recent monograph \cite{kn:nosotros2}.
We work with groups and varieties defined over an algebraically
closed field that will be denoted as $\Bbbk$.

If $G$ is an affine algebraic group then the algebra $\Bbbk[G]$ of
polynomial functions on $G$ (with pointwise operations of sum and
product) is in fact a Hopf algebra and its operations are defined as
follows.  The comultiplication $\Delta:\Bbbk[G] \to \Bbbk[G] \otimes
\Bbbk[G] \text{ written as } \Delta(f)=\sum f_1 \otimes f_2$
--Sweedler's notation-- is characterized by the fact that for all $x,y
\in G$: $\sum f_1(x)f_2(y)=f(xy)$. The antipode $S: \Bbbk[G] \to
\Bbbk[G]$ is defined for all $x \in G$ as $S(f)(x)=f(x^{-1})$ and the
counit $\varepsilon: \Bbbk[G] \to \Bbbk$ is $\varepsilon(f)=f(e)$.  In
particular de left and right translations of $f$ by an element $x \in
G$ are $x \cdot f =\sum f_1f_2(x)$; $f \cdot x =\sum f_1(x)f_2$.

A --not necessarily finite dimensional-- rational (left) $G$--module
$M$ can be defined in terms of a (right) $\Bbbk[G]$--comodule
structure $\chi_M: M \to M \otimes \Bbbk[G]$, and this structure map is
written \`a la Sweedler as $\chi(m)=\sum m_0 \otimes m_1 \in M \otimes
\Bbbk[G]$. It is related with the action of $G$ on $M$ by the formula
($x \in G\,,\, m \in M$): $x \cdot m=\sum m_0m_1(x)$.  The category of
rational $G$--modules is denoted as $\modG$, and by definition it
coincides with the category of $\Bbbk[G]$--comodules. If $N \in
\modG$, we denote as ${}^GN:=\{n \in N: x\cdot n=n \text{ for all } x
\in G\}$ and it is clear that ${}^GN=\{n \in N: \chi(n)=n \otimes 1\}$
with $\chi$ the $\Bbbk[G]$--comodule structure on $N$. If $M$ is a finite dimensional rational $G$--module and $m \in M\,,\,\alpha \in M^*$ we call $\alpha|m \in \Bbbk[G]$ the polynomial $\alpha|m=\sum \alpha(m_0)m_1$ or in explicit terms: $(\alpha|m)(x)=\alpha(x\cdot m)$ for $x \in G$. It is clear that $x\cdot (\alpha|m)=\alpha|(x \cdot m)$ for all $x \in G$. Also, in the
case of a closed inclusion $H \subseteq G$ of affine algebraic groups,
if $N \in \modG$, $N|_H$ is the $H$--module obtained by result of the
restriction of the $G$--action to an $H$--action. In this situation if
the structure of $\Bbbk[G]$--comodule of $N$ is $\chi(n)=\sum n_0
\otimes n_1 \in N \otimes \Bbbk[G]$, the structure of $N|_H$ as a
$\Bbbk[H]$--comodule is $(\id \otimes \pi)\chi(n)=\sum n_0 \otimes
\pi(n_1) \in N \otimes \Bbbk[H]$ where $\pi:\Bbbk[G] \to \Bbbk[H]$ is
the restriction morphism. 

Concerning some algebraic aspects: all algebras will be commutative
--unless explicitly stated-- and over a base field $\Bbbk$ that is
algebraically closed. An algebra is affine if it is commutative,
finitely generated and with no non--zero nilpotents.

Morever, the category $\modG$ for an affine algebraic group $G$ is abelian, and has enough injectives. This guarantees that the basic machinery of homological algebra is available in the working platform of this survey.
In particular, this category has the particularity that $\Bbbk[G] \in \modG$ is an injective object and also that if $M \in \modG$ is an arbitrary rational $G$--module, then $M \otimes \Bbbk[G]$ is injective. In this manner one has that the coaction map $\chi:M \to M \otimes \Bbbk[G]$ produces an imbedding of $M$ in an injective object and this guarantees that the category has enough injectives.

Sometimes we deal with the categories of $(R,G)$--modules --denoted as
$\modRG$, where $R$ is a rational commutative $G$--module algebra. 
We say that $M$ is an $(R,G)$--module, provided that it is a rational $G$--module, a module over
  the ring $R$ and that the actions are related in the
  following manner if $x \in G\,,\, r \in R\,,\,m\in M$,
  $x\cdot (rm)=(x\cdot r)(x \cdot m)$. The morphisms are defined in
  the obvious way.
  
\section{Antecedents, faithfull representations of Lie groups}

The concerns that led to the discovery of the concept of
observability, seem to derive from the persuit of the understanding
and simplification of a series of results on the existence of faithful
finite dimensional representations of Lie groups (due to E. Cartan,
M. Goto, D. Ado, A. Malcev, K. Iwasawa, G. Hochschild and others).

Below we trace backwards the main steps of this process.

Previously to the results appearing in ERA, Hochschild and Mostow
published in 1957/58 two important papers
(\cite{kn:homolie,kn:mostow}) on the extension of representations of
Lie groups that are cited explicitly in the aforementioned introduction
of ERA:
\smallskip
\begin{center}
\begin{minipage}{.8\textwidth}
  {\small {\em
In the analogous situation
  for Lie groups, an analysis of the the obstructions to the
  extendibility of representations of a subgroup has been made only
  for normal subgroups, \cite{kn:homolie,kn:mostow}, and not
  much is known in the general case. The algebraic case turns out to
  be much more accessible.}}
\end{minipage}
\end{center}
\smallskip
The differences between the algebraic case and the Lie group situation
are remarkable and it is patent from the comparison between the results for Lie groups in  \cite{kn:homolie,kn:mostow} and the situation of algebraic groups in \cite{kn:hmbb}.

For example, in the first mentioned papers and in a rather laborious
way, the authors prove the following result.
\begin{teo}\cite[Theorem 4.1]{kn:mostow} Let $H
\subseteq G$ be a closed normal inclusion in the category of (real or
complex) analytic groups and denote as $N$ the radical of the
commutator subgroup $G'$ of $G$. Assume that $\rho$ is a finite
dimensional representation of $H$ and that $\rho'$ is the semisimple
representation associated to $\rho$.  Then, $\rho$ can be extended to
$G$ (with a finite dimensional extension) if and only if the following
three conditions hold:
\begin{enumerate}
\item $\rho'$ is trivial in $H\cap N$;
\item The representation $\sigma$ of $HN$ defined by
  $\sigma(xu)=\rho'(x)$ for $x \in H\,,\, u \in N$ is continuous when
  $HN$ is endowed with the topology induced by $G$;
\item Call $G_{f}$ the intersection of all the kernels of all the
  finite dimensional representations of $G$. Then $\rho$ is trivial in
  $G_f \cap H$.
\end{enumerate}
\end{teo}

The above theorem is the main result of \cite{kn:mostow}, whereas in
the first paper \cite{kn:homolie}, a particular case is proved with
additional topological conditions. It is interesting to compare it
with the following very simple criterion for the extension of a
representation in the case of affine algebraic groups without the hypothesis of normality (this subject
will be treated in more detail and precision in Section
\ref{subsection:obsgeo}).

First we need to introduce some definitions.
\begin{defi} Let $H \subseteq G$ be
  a closed inclusion of affine algebraic groups. \begin{enumerate}
  \item A character $\chi:H \to \Bbbk$ is said to be extendible to $G$
    if there is a polynomial function $f \in \Bbbk[G]$ such that
    $f(1)=1$ and for all $x \in H$, $x \cdot f =\chi(x)f$ --or
    equivalently, for all $y \in G$, $f(yx)=f(y)\chi(x)$.
    \item If $M=(M,\cdot)$ is a rational $H$--module, and $\chi$ is a
      character of $H$, we call $M_\chi$ the rational $H$--module
      $(M,\cdot_\chi)$ where $\cdot_\chi $ is defined on $M$ as $x
      \cdot_\chi m=\chi(x)(x\cdot m)$ for all $m \in M$. Clearly
      $M_\chi= M \otimes \Bbbk_\chi$ where $\Bbbk_\chi$ is the one
      dimensional $H$--module associated to the character $\chi$.
    \end{enumerate}
    \end{defi}
Next theorem guarantees that for affine algebraic groups, every
representation can be extended ``up to the twist by an extendible
character''.

\begin{theo}\cite[Theorem 1]{kn:hmbb}\cite[Theorem 8.2.3]{kn:nosotros2}
  \label{theo:obse1} If $H \subseteq G$
  is a closed inclusion of affine algebraic groups for any rational
  finite dimensional $H$--module $M$ there exists a finite dimensional
  rational $G$--module $N$ and a character $\chi: H \to \Bbbk$ such
  that:
  \begin{enumerate} \item The character $\chi$ is extendible;
    \item $M_\chi \subseteq N|_H$, where $N|_H$ denotes that we
      consider the action of $N$ restricted to $H$.
\end{enumerate}
      Moreover, in the case that $M$ is a simple $H$--module, $N$ can be taken to be a simple $G$--module and even more particulary a simple $G$--submodule of $\Bbbk[G]$. Also given a pair $0 \neq m \in M$ and $z \in G$ there is such an injection $M_\chi \subseteq \Bbbk[G]$ such that $m(z) \neq 0$.  
  
\end{theo}
\begin{theo}\cite[Theorem 11.2.9]{kn:nosotros2}\label{theo:obse2}
In the situation above, if the character $\chi^{-1}$ is extendible
then the finite dimensional $H$ module $M$ can be imbedded (as a $H$
submodule) in a finite dimensional $G$--module $N$. In particular if for every extendible character $\chi$, the character $\chi^{-1}$ is also extendible the subgrup $H$ is observable in $G$.  Moreover if $H \subseteq G$ is observable, then all characters of $H$ are extendible to $G$.  
\end{theo}
\begin{proof} Imbed first $M_\chi \subseteq N|_H$
  and then consider the inclusion of $H$--modules $\Bbbk\chi^{-1} \to
  \Bbbk[G]$ that sends $\chi^{-1} \mapsto f$ where $f$ is the
  polynomial guaranteeing the extendibility of $\chi^{-1}$. Clearly
  the tensor products of the corresponding maps gives an inclusion of
  $H$--modules from $M:=M_\chi \otimes \Bbbk \chi^{-1} \to (N \otimes
  \Bbbk[G])|_H$. As the image of $M$ inside of $N \otimes \Bbbk[G]$
  lies in $N \otimes \Bbbk f$, that is finite dimensional rational
  $G$--module, the proof of the first assertion is finished. The
  second assertion follows directly from the first. It only remains to
  prove that if $\chi$ is an arbitrary character of an observable $H$,
  then $\chi$ is exendible. Given $\chi$, an arbitrary character of
  $H$, we can find a finite dimensional $G$--module $N$ and an
  $H$--inclusion of $\Bbbk_\chi \to N$. If we call $n \in N$ the image
  of $\chi$, we have that for all $x \in H\,,\, x\cdot n= \chi(x)n$.
  If $\alpha \in N^\vee$ is a linear functional such that
  $\alpha(n)=1$ and take the polynomial $\alpha|n \in \Bbbk[G]$
  (recall that $(\alpha|n)(y)=\alpha(y\cdot n)$ for all $y \in G$). It
  is clear that if $x \in H$ then $(x \cdot
  (\alpha|n))(y)=(\alpha|n)(yx)=\alpha((yx)\cdot n)= \alpha(y\cdot (x \cdot n))=\alpha(y
  \cdot (\chi(x)n))=\chi(x)\alpha(y \cdot n)= \chi(x) (\alpha|n)(y)$. Moreover, $(\alpha|n)(1)=\alpha(n)=1$.

  \end{proof}

It seems that the main motivation of the authors of \cite{kn:homolie,kn:mostow} to study the extension of representations from normal Lie subgroups to the whole group, was the search for the simplification and unification of some of the proofs of the standard results on faithfull representations of Lie groups. In this respect, in the
introduction to \cite{kn:mostow} and after describing the
main results of \cite{kn:homolie} the author writes:
\smallskip
\begin{center}
\begin{minipage}{.8\textwidth}
  {\small {\em From the extension [results of \cite{kn:homolie}...]
      one deduces quickly all the standard results on faithful
      representations of Lie groups.}}
\end{minipage}
\end{center}
\smallskip

Indeed, in \cite[Section
  3]{kn:homolie}, short new proofs of the following three classical and
important theorems are presented. E. Cartan's theorem on the existence
of a faithful representation of a simply connected solvable Lie group,
that is unipotent in a maximal normal nilpotent subgroup; Goto's
theorem on the existence of a faithful represention of a connected Lie
group $G$ provided we know the existence of a representation for a
maximal semisimple subgroup together with additional topological
conditions on the radical of the commutator subgroup of $G$, and Malcev
theorem that guarantees the existence of a faithfull representation of
a connected Lie group once we know that such a representation exists
for the radical of $G$ and for a maximal semisimple analytic subgroup
of $G$.

\section{Observability and geometry, observability and invariant theory}
\subsection{Observability and geometry}\label{subsection:obsgeo}

One of the more interesting results of ERA is the discovery of the
relationship between the extension of the representations from $H$ to
$G$ and the geometric structure of the homogeneous space $G/H$.

It is substantially harder to study homogeneous spaces in the category
of algebraic groups than for example in the closely related category
of Lie groups.  The basic general results concerning the existence of
a natural structure of algebraic variety on $G/H$ are due to
M.~Rosenlicht and A.~Weil in the mid 1950s (see \cite{kn:Roshomo} and
\cite{kn:Weilhomo}).  The proof that $G/H$ is quasi--projective is due
to W.~Chow and appeared in 1957 (see \cite{kn:ChowPQ}).

The proof  that the quotient of an affine group by a normal closed
subgroup is also an \emph{affine} algebraic group seemed to have appeared for the first time in
1951\footnote{Probably it was known to specialist since the beginning
  of the theory.}, in Chevalley's very important foundational book,
\cite{kn:Che}.

In ERA the following theorem --that provides a very precise characterization of
observability in geometric terms-- is proved.

\begin{theo}\label{theo:geocarac}\cite[Theorem 4]{kn:hmbb} If $H \subseteq G$ is a closed inclusion of affine algebraic groups, then $H$ is observable in $G$ if and only
  if the homogeneous space $G/H$ is a quasi--affine variety.
 \end{theo}

In particular, the above theorem guarantees that a normal subgroup is
always observable and hence, that the normality hypothesis
unavoidable for the situation of Lie groups as presented in \cite{kn:homolie,kn:mostow}, is unnecesary in the
category of algebraic groups.

For the proof of Theorem \ref{theo:geocarac} we need some preparation.
\begin{lema}\label{lema:gene} Assume that $H \subseteq G$ is a closed inclusion of affine algebraic groups, if $0 \neq I \subseteq \Bbbk[G]$ is an $H$ stable ideal, there is a non zero element $f \in I$ and an extendible character $\chi$ of $H$ such that $x\cdot f = \chi(x)f$ for all $x \in H$ and that $f(1)\neq 0$. 
\end{lema}
\begin{proof} Take $V$ a simple rational $H$--submodule of $I$. Choose a basis $\{e_1,\cdots,e_n\}$ with the property that $e_1(1)=1,e_i(1) = 0$ for $i=2,\cdots,n$. Take $V^*$ the linear dual of $V$ that is a rational $H$--module, and apply Theorem \ref{theo:obse1} to obtain an extendible character $\chi$ of $H$ and inclusion $\iota : V^* \to \Bbbk[G]_{\chi^{-1}}$ with the property that $\iota(e_1^*)(1)\neq 0$. The equivariance property of $\iota$ reads as: $\iota(x\cdot\alpha)=x\cdot_{\chi^{-1}} \iota(\alpha)=\chi^{-1}(x)x\cdot \iota(\alpha)$. It is clear
  that $\sum e_i \otimes e_i^*$ is $H$--stable, i.e. for all $x \in H$
  we have that $x\cdot e_i \otimes x\cdot e_i^*= \sum e_i \otimes
  e_i^*$, if we apply $\id \otimes \iota$ to this equality, we deduce
  that $\sum e_i \otimes \iota(e_i^*)=\sum x\cdot e_i \otimes
  \iota(x\cdot e_i^*)= \chi^{-1}(x)\sum x\cdot e_i \otimes x \cdot
  \iota(e_i^*)$.  Then, the element $f = \sum e_i\iota(e_i^*) \in I$
  satisfies the following equivariance property $f=\sum e_i
  \iota(e_i^*)=\chi^{-1}(x)\sum (x\cdot e_i) (x \cdot
  \iota(e_i^*))=\chi^{-1}(x)x\cdot \sum e_i \iota(e_i^*)=\chi^{-1}(x)
  x \cdot f$. Moreover, $f \neq 0$ as $f(1)=\sum e_i(1)
  \iota(e_i^*)(1) = e_1(1)\iota(e_1^*)(1)=\iota(e_1^*)(1)\neq 0$.
  \end{proof}
The following characterization of observability seems to have appeared
for the first time in \cite[Chapter 11, Section 5]{kn:nosotros2}.

\begin{theo}{\cite[Theorem 11.5.1]{kn:nosotros2}} \label{theo:charcinv} Assume that $H \subseteq G$ is a closed inclusion of affine algebraic groups. Then $H$ is observable in $G$ if and only if, for every $H$--stable ideal $I \subseteq \Bbbk[G]$, there is a non zero element $f \in I$ such that $x \cdot f =f$ for all $x \in H$.
Also, $H \subseteq G$ is observable if and only if for
every closed proper subset of the quotient space $C \subseteq G/H$
there is a non zero invariant polynomial such that $f(C)=0$.
\end{theo}
\begin{proof} In accordance with the lemma just proved, we can find $h \in I$ with the property that $x\cdot h = \chi(x)h$ for some extendible character $\chi$ and with $h(1)=1$. If $H$ is observable, the character $\chi^{-1}$ is also extendible and then there is an element $g \in \Bbbk[G]$ such that $g(1)\neq 0$ and $x \cdot g=\chi^{-1}(x)g$ for all $x \in H$. Hence $hg \in I$ and is $H$--invariant and not zero. For the converse, if we have an extendible character $\chi$, we have an element $0 \neq f \in \Bbbk[G]$ that is $\chi$-semi invariant and the associated principal ideal $I=\Bbbk[G]f$ is not zero and $H$--stable. By hypothesis, we can find
  $hf=g \in I$ with the property that $fg=h=x\cdot h=(x\cdot f) (x
  \cdot g)=(\chi(x)f )(x\cdot g)$. If $G$ is connected we can cancel $f
  \neq 0$ and we have that $x \cdot g=\chi^{-1}(x)g$. So that
  $\chi^{-1}$ is extendible and in accordance with Theorem
  \ref{theo:obse2}, the group $H$ is observable. The case that $G$ is
  not connected can be proved following the same methods.
The second assertion is basically a reformulation of the first
    (ideal--theoretical) characterization of observability in geometric terms.
\end{proof}

\begin{obse}\label{obse:obse3} If $H$ is observable in $G$ --with $G$ connected-- then
    ${}^H[\Bbbk[G]]=[{}^H\Bbbk[G]]$ (compare with \cite[Lemma
    11.5.4]{kn:nosotros2} where this result is proved and also the
  converse). It is clear that in general $[{}^H \Bbbk[G]] \subseteq
  {}^H[\Bbbk[G]]$. Conversely, take $0\neq g \in {}^H[\Bbbk[G]]$ and
  consider the $H$--stable ideal $I_g=\Bbbk[G]g \cap \Bbbk[G]$.  In
  accordance with Theorem \ref{theo:charcinv} we can find a non zero
  polynomial $f_1$ in $I_g$ that is also $H$--fixed. It we write
  $f_1=f_2g$ for $g$ as above, using the fact that $f_1$ and $g$ are
  fixed by $H$, we conclude that $f_2$ is also fixed.

  \end{obse}

Next we prove Theorem \ref{theo:geocarac}.

\bigskip
    {\em Proof of Theorem \ref{theo:geocarac}:}

    From the following general fact (see \cite[Theorem 1.4.48]{kn:nosotros2}): if $C \subseteq X$ is a closed subset of a quasi--affine variety,
    then there is a global section $0\neq f \in \mathcal O_X(X)$ such
    that $f|_C=0$, and the second assertion of Theorem \ref{theo:charcinv}, it follows directly that if $G/H$ is quasi affine then $H$ is observable in $G$. 

    We sketch the proof of the converse assertion and we work in the case that $G$ is irreducible (it is easy to show that is enough to treat this particular situation).

    Assume that $H \subseteq G$ is observable, and using the fact that
    ${}^H[\Bbbk[G]]=[{}^H\Bbbk[G]]$ (see Observation \ref{obse:obse3})
    we can take a family of field generators of the invariant rational
    functions $\{f_1,\cdots,f_n\} \subseteq
    {}^H[\Bbbk[G]]=[{}^H\Bbbk[G]]$ that are of the form $f_i=u_i/u_0$
    with $\{u_0,\cdots,u_n\} \subseteq {}^H\Bbbk[G]$ for
    $i=1,\cdots,n$. Let $N$ the finite dimensional rational
    $G$--module generated by $\{u_0,\cdots,u_n\}$,
    $M=\bigoplus_{i=0}^nN$ and take $m_0=(u_0,\cdots,u_n) \in M$.  It
    is a standard result in the theory of affine algebraic groups that
    $H=\{x \in G: x\cdot f=f \text{ for all } f \in {}^H[\Bbbk[G]]\}$
    (see for example \cite[Corollary 8.3.4]{kn:nosotros2}) and then in
    our case we have that $H=G_{m_0}$ the stabilizer of $m_0$. It can
    be proved that $G/H$ is isomorphic to the $G$--orbit of $m_0$ in
    $M$ (result that is obvious in the case of zero characteristic,
    but that in general a proof of the separabililty of the action in
    this situation is needed) and as such it is a quasi--affine variety
    (for more details see \cite[Section 8.3]{kn:nosotros2}).
    \qed
    
    \subsection{Observability and Hilbert's $\boldsymbol{14}^{\text{th}}$ problem}

About ten years after the introduction of the concept of
observability, an important relation with the so called Hilbert's
$14^{\text{th}}$ problem was discovered by G. Grosshans in
\cite{kn:gross14}. As such, the concept of observability became another important element in the toolkit of invariant theory.

We describe briefly some parts of the contents of the important paper
mentioned above, wherein the author distinguishes three situations
--that he names as ``the main problems''\footnote{We set the problems
  --specially the second and third-- in a slightly more general
  context than the original one due to Grosshans. For this we follow
  basicaly the presentation of \cite[Chapter 11,13]{kn:nosotros2}.}.

\smallskip
\noindent {\tt Problem 1. Galois characterization of the observable subgroups.} 

The author presents an interesting new perspective of the concept of
observable subgroup.

\begin{defi}\label{defi:viagalois}
  If $G$ is an affine connected algebraic group, define the sets
  $\mathfrak H=\{H: H \subseteq G \text{ is a closed inclusion}\}$ and
  $\mathfrak R=\{R \subseteq \Bbbk[G]: R \text{ is a }
  \Bbbk\text{--subalgebra of } \Bbbk[G]\}$ and the maps:
  \begin{enumerate}
  \item $\mathbb F: \mathfrak H \to \mathfrak R\,,\, \mathbb
    F(H):={}^H\Bbbk[G]=\{f \in \Bbbk[G]: x\cdot f=f, \forall x \in
    H\}$;
    \item $\mathbb S: \mathfrak R \to \mathfrak H\,,\, \mathbb S(R)=\operatorname{Stab}(R):=\{x  \in \Bbbk[G]: x\cdot r=r, \forall r \in R\}$.
    \end{enumerate}
\end{defi}

In the above situation it is usual to write $\mathbb F(H)=H'$ and similarly, $\mathbb S(R)=R'$. 
\begin{theo}\label{theo:connection}
  In the above context, if we endow the sets $\mathfrak
  H\,,\,\mathfrak R$ with the order given by inclusion, the maps
  $\mathbb F, \mathbb S$ form an (order inverting) Galois
  connection. Moreover, the fixed subgroups for this connection,
  i.e. $\{H \in \mathfrak H: H''=H\}$ are the observable subgroups.
\end{theo}

\begin{coro}\label{coro:firstproperties}
  In the above situation one has that:
  \begin{enumerate}
  \item For any $H \subseteq G$, $H''$ is observable in $G$;
  \item $H''=\bigcap \{K: H \subseteq K \subseteq G\,,\, K \text{
      observable}\}$;
  \item If $A$ is a commutative rational $G$--module algebra, and $H \subseteq G$ is a closed inclusion, then ${}^H A={}^{H''}A$.
\end{enumerate}
  \end{coro}
\begin{proof}
  \begin{enumerate}
  \item  For a Galois connection $H'''=H$ hence (1);
  \item In the above situation $H \subseteq K$ implies that  $H''\subseteq K''=K$ and then $H'' \subseteq \bigcap\{K: H \subseteq K \subseteq G\,,\, K \text{ observable}\}$ and being $H''$ observable, the proof of this part is finished.
    \item It is clear that ${}^{H''}A \subseteq {}^H A$. Take now, $a \in {}^H A$ and let $V$ be a rational finite dimensional $G$--module that contains $a$. If we call $K=\{x \in G: x\cdot a = a\}$, then Theorem \ref{theo:follows} guarantees that $K$ is observable. As $a$ is fixed by $H$ we deduce that $H \subseteq K$ and then $H'' \subseteq K$ and that means that $a$ is fixed by all the elements of $H''$. 
    \end{enumerate}
\end{proof}
\begin{theo}{\cite[Theorem 8]{kn:hmbb},\cite{kn:gross14}}\label{theo:follows} Assume that $H \subseteq $ is a closed inclusion of affine algebraic groups and that there is a finite dimensional rational $M \in \modG$ with the property that there exists $m_0 \in M$ such that $H=G_{m_0}$. Then $H$ is observable in $G$.
  \end{theo}
\proof Take an arbitray $\alpha \in M^*$. The element $\alpha|m_0$
satisfies the following equivariance condition for all $x \in G$ $x
\cdot (\alpha|m_0)= \alpha| (x \cdot m_0)$. Then, $\alpha|m_0 \in {}^H
\Bbbk[G]$. Assume now that $z \in G$ is such that $z \cdot f = f$ for
all $f \in {}^H\Bbbk[G]$. Then for all $\alpha$ we have that $z \cdot
(\alpha|m_0)=\alpha|m_0$ and this implies that $\alpha(z \cdot m_0)=\alpha(m_0)$ for all $\alpha \in M^*$. Then $z \cdot m_0=m_0$ and then $z \in H$. Hence $H=\{z \in G: z \cdot f = f \text{ for all } f \in {}^H\Bbbk[G]\}$.
It is well known from the general theory of affine algebraic groups, that the above characterization of $H$ as the stabilizer of ${}^H\Bbbk[G]$ guarantees that $[{}^H\Bbbk[G]]={}^H[\Bbbk[G]]$. Along the proof of Theorem \ref{theo:geocarac} we proved that the condition $[{}^H\Bbbk[G]]={}^H[\Bbbk[G]]$ guarantees the observability of $H$ on $G$.  
\qed

\smallskip
\noindent {\tt Problem 2. Descent of the finite generation condition.}
\smallskip

Assume that $H \subseteq G$ is a closed inclusion of affine algebraic groups and that $A$ is a rational $G$--module algebra we say that the \emph{the finite generation condition descends from $G$ to $H$} if for all $A$ as above, in the inclusion ${}^G A \subseteq {}^HA$ the finite generation of the smallest $\Bbbk$--algebra implies the finite generation of the larger. 

It is natural to search for conditions for $G$ and $H$ for which the finite generation of invariants descends from $G$ to $H$.

The first thing to notice is that having $H''$ the same invariants
than $H$ we can assume without loss of generality that $H$ is
observable in $G$ as ${}^GA \subseteq {}^{H''}A = {}^HA$. This is a
crucial observation that reduces some problems in invariant theory to
the case of observable subgroups.

\begin{defi}
  Let $H \subseteq G$ be a closed inclusion we say that
  ``the pair $(H,G)$ satisfies the codimension two condition'' if there exists a finite dimensional rational $G$--module $V$ and an element $v \in V$ such that: (i) $H=\{x \in V: x\cdot v =v\}$ and $G/H \cong G\cdot v$; (ii) for each irreducible component $C$ of $\overline {(G \cdot v)} \setminus (G \cdot v)$, we have that $\operatorname{codim}_{\overline {G \cdot v}} C \geq 2$.
\end{defi}

In that context, the following theorem is proved in \cite{kn:gross14}:
\begin{theo}\cite[Theorem 4.3]{kn:gross}\label{theo:descent} For the situation above if $H$ is observable in $G$, the conditions:
  \begin{enumerate}
  \item The $\Bbbk$--algebra $H' \subseteq \Bbbk[G]$ is finitely
    generated;
  \item The pair $(H,G)$ satisfies the codimension two condition;
  \item For any finitely generated rational $G$--module; algebra A, $A^H$ is a finitely generated $\Bbbk$--algebra,
  \end{enumerate}
  are related as follows.
Conditions (1) and (2) are equivalent and condition (3) implies both of them. In the case that the action of $G$ on $V$ is separable and $G$ is reductive, the three conditions are equivalent. 
 \end{theo}

For a proof of this theorem we refer the reader to \cite{kn:gross} or
to a more recent exposition appearing in \cite[Section
  13.5,13.6]{kn:nosotros2}.

The so called ``codimension two condition'' is used in order to apply the following theorem on  the extension of regular functions.
``Let $X$ be an irreducible normal variety
and  $f \in \mathcal O_X(U)$  be a function defined in an open subset U such that $\operatorname{codim}_X(X \setminus U) \geq 2$, then $f$ can be extended to a function defined in X''. See \cite[Lemma 1]{kn:gross14} or \cite[Theorem 2.6.14]{kn:nosotros2} for (similar) proofs of this general result.

It is worth noticing that in case of the special hypothesis on the separability of the action, the ring $H'$ ``behaves like a universal object as far as finite generation is concerned'' (see \cite[page 231]{kn:gross14}). 

\medskip
\noindent {\tt Problem 3. Hilbert's $14^{\text{th}}$ problem.}
\smallskip

The original Hilbert's $14^{\text{th}}$ problem examines the answers to the following question
(see \cite{kn:hilbertpr}).
\medskip

{\tt Hilbert's problem.} {\em Let $A=\Bbbk[X_1,\cdots,X_n]$ be the
  polynomial algebra in $n$ variables, let $H$ be a subgroup $H
  \subseteq \operatorname{GL}_n(\Bbbk)$ and consider the action of $H$
  on $A$ given by the restriction of the natural action of
  $\operatorname{GL}_n(\Bbbk)$. Is the subalgebra of $H$--invariants
  of $A$ finitely generated?}

\medskip

This problem can be generalized to the following context.
\medskip

{\tt Generalized Hilbert's problem.} {\em Assume that
$H \subseteq G$ is a closed inclusion of affine algebraic groups, and that $A$ is a finitely generated commutative $\Bbbk$--algebra.
Assume that $G$ acts rationally in the affine algebra $A$. Find
conditions for the pair $(H,G)$ that guarantee that if $A^G$ is finitely generated so is $A^H$.}

\smallskip
It is clear that Theorem \ref{theo:descent} guarantees that if $G$ is
reductive then, the generalized Hilbert's $14^{\text{th}}$ problem has a positive
answer if $H$ is observable in $G$. 

\subsection{The perspective of observability in Hilbert's
  $\boldsymbol{14}^{\text{th}}$ problem}
The original formulation by D. Hilbert of his famous
$14^{\text{th}}$ problem
reads as follows (as it appeared translated into English in
\cite{kn:hilbertpr}): 

\begin{quote}
``By a finite field of integrality I mean a system
of functions from which a finite number of functions can be chosen, in
which all other functions of the system are rationally and integrally
expressible. Our problem amounts to this: to show that all relatively
integral functions of any given domain of rationality always
constitute a finite field of integrality''.
\end{quote}

In modern language this problem can be formulated as follows --see
\cite{kn:Mumsur}--: ``Let $\Bbbk $ be a field [$ \{x_{1},\dots,x_{n}\}$
a family of indeterminates] and let $K$ be a subfield of 
$\Bbbk (x_{1},\dots,x_{n})$: $\Bbbk  \subset K \subset
\Bbbk (x_{1},\dots,x_{n})$. Is the ring $K \cap \Bbbk [x_{1},\dots,x_{n}]$
finitely generated over $\Bbbk $?''.

This problem of the {\em finite generation of special subalgebras of
the polynomial algebra 
$\Bbbk [x_{1},\dots,x_{n}]$} is known as {\em
Hilbert's 14$^{\text{th}}$ problem}\/  because  it appeared with that
number in the  
list of 23 problems presented by Hilbert in the 
International Congress of Mathematicians celebrated in Paris in 1900
(\cite{kn:hilbertpr}).

A particularly important case  is the following:  

Lt $G\subset  \operatorname{GL}_{n}$ be a subgroup, consider the induced
action of $G$ on $\Bbbk [x_{1},\dots,x_{n}]$ and call
$K={}{^G}{\Bbbk (x_{1},\dots,x_{n})}{}$. As
${}{^G}{\Bbbk [x_{1},\dots,x_{n}]}{} = K 
\cap \Bbbk [x_{1},\dots,x_{n}]$, the finite generation of rings of
invariants could --in principle-- be deduced from an affirmative
answer to Hilbert's problem. 

In 1900, when Hilbert formulated his 14$^{\text{th}}$ problem, a few
particular cases were already solved. Classical invariant theorists
were concerned with the invariants of ``quantics'' (invariants for
certain actions of ${\operatorname{SL}}_{m}(\mathbb C)$). In this
situation the finite generation was proved by Gordan in 1868 for $m=2$
and by Hilbert in 1890 for arbitrary $m$. Hilbert mentioned as
motivation for his 14$^{\text{th}}$ problem work by Hurwitz and also
by Maurer--that turned out to be partially incorrect--.

Maurer's work contains some partial relevant results that were later
rediscovered by Weitzenb\"ock and guaranteed a positive answer for the
case of the invariants of $(\mathbb C, +)$ and $(\mathbb C^*,
\times)$. Later Weyl and Schiffer gave a complete positive answer for
semisimple groups over $\mathbb C$. More recently --based on the
platform established by Mumford in \cite{kn:Mumford}--, Nagata's
school contributions (see \cite{kn:nagataicm} and \cite{kn:nagatata}
together with Haboush's results (\cite{kn:haboushred}) settled the
question affirmatively for reductive groups over fields of arbitrary
characteristic.

In the case of non reductive groups, positive answers are more
scarce. It is worth mentioning --besides the contributions by Maurer and
Weitzenb\"ock for the case of the additive group of the field of
complex numbers-- a result by Hochschild and Mostow (valid in
characteristic zero):  if $U$ is the unipotent
radical of a subgroup $H$ of $G$ that contains a maximal unipotent
subgroup of $G$ then  the $U$--invariants of a finitely
generated commutative $G$--module algebra are finitely generated
(\cite{kn:homounip}). 

Around the same time of the publication of the paper just mentioned,
Grosshans' published the above mentioned papers that provide more
general insights into the problem of the finite generation of
invariants for a non reductive group in arbitrary characteristic. For
example the results of \cite{kn:homounip} can be understood as of
Grosshans' pairs and the same with the classical result of Maurer's
results on the invariants of the additive group. The so called
Popov--Pommerening conjecture concerning the finite generation of the
$U$--invariants of a finitely generated $G$--module algebra when $G$
is a reductive group and $U$ is a unipotent subgroup normalized by a
maximal torus of $G$ can also be formulated within that framework.  The
reader interested in these and many other topics in invariant theory
should read the survey \cite{kn:vinpop}.

It took almost 60 years  to discover
that, in general, the answer to Hilbert's 14$^{\text{th}}$
question is negative. The first counterexample was devised  by  
M.~Nagata and presented at the  International Congress of
Mathematicians in 1958 
(\cite{kn:nagataicm}). 
Nagata's counterexample consisted of a commutative unipotent 
algebraic group $U$ acting 
linearly and by automorphisms on a  polynomial
algebra, with a non finitely generated algebra of invariants.

\section{Observability, Integrals  and reductivity}

In 1977 in the article {\em Induced modules and affine quotients} (referred as IMAQ),
Cline, Parshall and Scott introduced a new viewpoint in the subject of
observability (see \cite{kn:CPS}) by relating it with homological concepts, such as the exactness of the induction functor and injectivity conditions. With hindsight we could say that in a non--explicit way, the idea of observability was related to a generalization of the concept of reductivity (see \cite{kn:firstwork,kn:secondwork}).

The authors summarize --rather succinctly-- the results of their paper as follows:
\smallskip
\begin{center}
\begin{minipage}{.8\textwidth}
{\small {\em Let $G$ be
  an affine algebraic group over an algebraically closed field
  $\Bbbk$. A closed subgroup $H$ of $G$ is exact if induction of
  rational $H$--modules to rational $G$-modules preserves short exact
  sequences.  The main result of this paper is that $H$ is exact iff
  the quotient variety $G/H$ is affine. (In case $G$ is reductive this
  means that $H$ is reductive.) Also, we obtain a characterization of
  exactness in terms of a strong observability criterion, in this
  respect our theorem generalizes a result of Bialynicki-Birula [2] on
  reductive groups in characteristic zero.}}
\end{minipage}
\end{center}

In the definition of strong observability, besides the existence of an
extension of an $H$--module $M$ by a $G$--module $N$, the authors
demand a condition that controls the relation between the
$H$--invariants of the submodule and the $G$--invariants of the
module.

The concept of exactness will be treated in detail in Section
\ref{section:obsadj}. Below we give the basic operative definition in
order to proceed as fast fast as possible to the main results.

\begin{defi} \label{defi:strongobs} Suppose that $H \subseteq G$ is a
  closed inclusion of affine algebraic groups.  We say that a rational
  $H$--module $M$ is strongly extendible, if there is a rational
  $G$--module $N$ such that $M \subseteq N|_H$ and ${}^HM \subseteq
  {}^GN$. If the pair $H \subseteq G$ is such that all rational
  $H$--modules are strongly extendible to $G$ we say that $H$ is
  strongly observable in $G$.
  \end{defi}

\begin{obse}  In the paper we are currently considering the authors
  write down a stronger condition for the fixed parts of the modules
  $N$ and $M$ in the above definition, they ask that ${}^HM = {}^GN$,
  but later in \cite[Remark 4.4,(c)]{kn:CPS} they comment that it can
  be weakened as above.
\end{obse}

\begin{defi}\label{defi:exactfirst} Assume that $H \subseteq G$ is a closed inclusion of
  affine algebraic groups.  We say that $H$ is exact in $G$ if for an
  arbitrary short exact sequence $0 \to P \to Q \to R \to 0$ of
  $(\Bbbk[G],H)$-- modules, the sequence $0 \to {}^HP \to {}^HQ \to
  {}^HR \to 0$ is exact.
  \end{defi}
Generalizing the relationship discovered in \cite{kn:hmbb}, between
the geometry of $G/H$ and the observability of $H$ in $G$, the authors
of \cite{kn:CPS} show that this more precise concept of ``strong
observability'', has relevant connections with: a. the geometric
structure of the homogeneous space $G/H$ (strengthenig the results
known for the observability situation); b. the exactness properties of
the induction functor from $H$--modules to $G$--modules; c. the
descent of the injectivity condition by restriction of the action.

Indeed, in \cite[Theorem 4.3, Proposition 2.1]{kn:CPS} the following
neat and comprehensive result is proved.

\begin{theo}\label{theo:mainCPS} For a closed inclusion of affine algebraic
  groups $H \subseteq G$, the following conditions are equivalent:
  \begin{enumerate}
  \item The subgroup $H$ is strongly observable in $G$.
    \item The rational $G$--module $\Bbbk[G]$ is injective when
      considered as an $H$--module. More generally for every injective
      rational $G$--module $I$, then $I|_H$ is also injective.
    \item The subgroup $H$ is exact in $G$\footnote{It is also usual
      to define this exactness in terms of the induction functor:
      $H$ is exact in $G$ if the induction functor $\operatorname{Ind}_H^G:\modH \to \modG$ is
      exact (see Definition \ref{def:induction})).}.
      \item The homogeneous space $G/H$ is affine.
    \end{enumerate}
  \end{theo}
The fact that (4) implies (3) was proved (as it is mentioned in the paper) --almost at the same time but
using different methods-- in Haboush's paper
\cite{kn:matsuhab}. Also another proof appeared around the same time in \cite{kn:rich}\footnote{See also the discussion later in the paper in Section \ref{subsection:soandred}.}. Moreover, in the introduction of \cite{kn:CPS}, it
is mentioned that the equivalence of (3) and (4) had been conjectured
by J.A. Green before.

\subsection{Strong observability, injetivity and integrals}

We deal next with the first two conditions of Theorem
\ref{theo:mainCPS} leaving the third and fourth for later
consideration. Our proofs will be different from the orginals as we
use ``integral tools''. Given an affine algebraic group $H$ we define
the notion of integral in $H$ (or $\Bbbk[H]$) with values in an
$H$--algebra $R$ and show the relation of integrals with strong
observability. This relation is implict in \cite[Theorem 3.1]{kn:CPS}
where the authors consider the strong observability for the situation
that $H$ unipotent.  Therein the authors mention \cite[Proposition
  2.2]{kn:ho} as an antecedent where the integrals appear as
cross--sections --in the same manner than in IMAQ--.

\begin{defi}\label{defi:integral}
\begin{enumerate}
\item An (scalar) integral for an affine group $H$ is a linear map
  $\sigma:\Bbbk[H] \rightarrow \Bbbk$ that is invariant
  --i.e. $\sigma(x\cdot f)= \sigma(f)$ for $x \in H$ and $f \in
  \Bbbk[H]$--. It is said to be total if $\sigma(1)=1$.
\item An integral with values in a rational $H$--module algebra $R$ is
  a linear map $\sigma:\Bbbk[H] \rightarrow R$ that is
  $H$--equivariant --i.e. $\sigma(x \cdot f)=x \cdot \sigma(f)$)--.
  We say that it is total if $\sigma(1)=1$.
  \end{enumerate}
\end{defi}

The relation of integrals with strong observability is deployed
explicitly in \cite[Theorems 11.4.8, 11.4.10]{kn:nosotros2} that we
write below and that guarantee the equivalence of conditions (1) and
(2) of Theorem \ref{theo:mainCPS}.

\begin{theo} \label{theo:intobsinj} Given the closed inclusion $H \subseteq G$, $H$ is strongly observable in $G$ if 
and only if $H$ admits a total integral with values in $\Bbbk[G]$ and
this happens if and only if $\Bbbk[G]$ is injective as an $H$--module.
\end{theo}
\begin{proof} First we prove the equivalence of the injectivity condition with the existence of a total integral. 
  If $\Bbbk[G]$ is injective in $\modH$, we can complete the diagram
\[
\xymatrix{ \Bbbk \ar[d] \ar[r] & \Bbbk [H] \ar@{.>}[dl]^{\sigma}
  \\ \Bbbk[G] & }
\]
and produce a morphism of $H$--modules $\sigma :\Bbbk [H] \to
\Bbbk[G]$, sending $1$ into $1$.

Conversely, assume that $\sigma: \Bbbk [H] \to \Bbbk[G]$ is a total
integral and define the map $\Lambda:\Bbbk[G] \otimes \Bbbk [H] \to
\Bbbk[G]$ by the formula for $r \in \Bbbk[G] \text{ and } f\in
\Bbbk[H]$, $\Lambda(r \otimes f)=\sum r_{1}\sigma\bigl(S(\pi(r_{2}))
f\bigr)$ where $\Delta(r)= \sum r_1 \otimes r_2 \in \Bbbk[G] \otimes
\Bbbk[G]$. If $\chi(r)=\sum r_1 \otimes \pi(r_2)$ is the $H$--comodule
structure map for $\Bbbk[G]$, then $(\Lambda \chi)(r)=\sum r_{1}
\sigma \bigl(S(\pi(r_{2}))\pi(r_{3})\bigr) = r\sigma(1)=r$. Also, if
$r \in \Bbbk[G]$ and $x \in H$, then $\sum x\cdot r_{1} \otimes
\pi(r_{2})\cdot x^{-1}=\sum r_{1} \otimes \pi(r_{2})$, equality that
can be proved directly by evaluation of both sides at an element
$(y,z) \in G \times H$ (the left and right side yield the value
$r(yz)$ after evaluation).  Then for all $x \in H$,
\begin{align*}
\Lambda(r \otimes x\cdot f) =&\ \sum r_{1}
\sigma\bigl(S(\pi(r_{2}))(x\cdot f)\bigr) = \sum (x\cdot r_{1})
\sigma\bigl(S(\pi(r_{2})\cdot x^{-1})(x\cdot f)\bigr)=\\ & \ \sum
(x\cdot r_{1})\sigma \bigl(x\cdot\bigl(S(\pi(r_{2})) f\bigr)\bigr)=
x\cdot \sum r_{1} \sigma \bigl(S(\pi(r_{2})) f\bigr)=\\ & \ x\cdot
\Lambda(r \otimes f)
\end{align*}

If we write as $\Bbbk[G]_0\otimes \Bbbk[H]$ the rational $H$--module
with trivial $H$--action in the first tensor factor and the regular
action on the second, the above considerations show that $\chi:
\Bbbk[G]\to \Bbbk[G]_0\otimes \Bbbk[H]$ splits the $H$--morphism
$\Lambda : \Bbbk[G]_0\otimes \Bbbk[H]\to \Bbbk[G]$. Hence, $\Bbbk[G]$
is a direct $H$--module summand of $\Bbbk[G]_0 \otimes \Bbbk [H]$ and
hence (as it is well known that $\Bbbk[H]$ is injective as a rational
$H$--module) the polynomial algebra $\Bbbk[G]$ is also injective as a
rational $H$--module.
\medskip

Next we show how to produce a total integral if we know that the
inclusion $H \subseteq G$ is strongly observable.

Assume that $H$ is strongly observable and consider the $H$--module
$\Bbbk[H]$. By the hypothesis of strong observability, one can find an
inclusion $\Bbbk[H] \subset N$ where $N$ is a rational $G$--module and
${}^H\Bbbk[H]=\Bbbk= N^G$.  Take a linear functional $\alpha$ on $N$
such that $\alpha(1)=1$ and define $f \mapsto \sigma(f):
\Bbbk[H]\rightarrow \Bbbk[G]$ as: $\sigma(f)(x)= \alpha(x \cdot f)$
for $x \in G$. The integral is total as $\sigma(1)(x)=\alpha(x \cdot
1)=\alpha(1)=1.$
  
For the proof of the $H$--equivariance of $\sigma$ we compute
$\sigma(y\cdot f)(x)= \alpha(x \cdot y \cdot f)=\alpha(xy \cdot
f)=\sigma(f)(xy)=(y \cdot \sigma(f))(x).$

We finish the proof of the theorem by showing that the existence of a
total integral implies the strong observability of $H$ in $G$.

Assume that $\sigma$ is a total integral. First show that $H$ is
observable in $G$. Assume that $\gamma$ is a rational character of $H$
and fix an $f \in \Bbbk[G]$ with the property that
$\pi(f)=f|_H=\gamma$. Define the following element of $\Bbbk[G]$:
$g=\sum\sigma\left(S(\pi(f_2))\gamma\right)f_1 \in \Bbbk[G]$. A direct
computation shows that for all $x \in H$ we have that $x\cdot
g=\gamma(x)g$. As $g(1)=1$ we conclude that $g$ extends $\gamma$ and
being $\gamma$ an arbitrary character we deduce the observability of
$H$ in $G$.

In order to prove that the observability is strong we proceed as
follows. Given $M \in \modH$ we take $S=\bigoplus S_i $ the socle of
$M$, $S_i$ a simple object in $\modH$. Using the fact that $H$ is
observable, and $S_i$ simple it is easy to show that we can find
$H$--equivariant inclusions $\eta_i:S_i \to T_i$ with $T_i$ a
$G$--module, and $\eta_i({}^HS_i)\subseteq {}^GT_i$. Then we have a
map $\eta:S \to \bigoplus T_i$ with the required property for the
strong observability. In other words, we have proved that if $H$ is
observable in $G$, an arbitrary rational $H$--module has its socle
strongly extendible to a $G$--module.  We go one step further and
prove that in our case, this $G$--module (that we call $L$) can be
taken to be injective. This is done by imbeding the $G$--module thus
obtained, using the structure map $\chi: L \to L \otimes
\Bbbk[G]$. This map is equivariant when $G$ acts trivially in the
first tensor component, and using the fact that we have a total
integral, we see that $L \otimes \Bbbk[G]$ is injective as an
$H$--module.  All in all, we have proved that the original $H$--module
$M$ has its socle $S$ strongly extended to a $G$--module $M$ that is
injective as an $H$--module. The injectivity of $M$ guarantees the
extension of the map from $S$ to $M$ and this extension does the job
without increasing the $H$--invariants as ${}^HS={}^HM$.
\end{proof}
  
  \subsection{Integrals, observability and invariants}

Here we describe briefly some aspects on the development of the ideas
concerning total integrals mainly in the context of algebraic groups.

It was realized around 1961 that the concept of ``integral'' taking
values in an arbitrary $\Bbbk [H]$--comodule algebra (or rational
$H$--module algebras) instead of in the base field $\Bbbk $ could be a
relevant tool to control the representations and the geometry of the
actions of the group $H$. A particularly interesting case is when the
$\Bbbk [H]$--comodule algebra is $\Bbbk [G]$ for $G$ an affine
algebraic group and $H$ a given subgroup.

An important motivation was the following. In \cite{kn:ho} and
\cite{kn:inj}, Hochschild set the basis of the cohomology theory of
affine algebraic groups --rational cohomology. It was soon observed
that if $G$ is an affine algebraic group and $H\subseteq G$ a normal
closed subgroup, then it was necessary to prove that $\Bbbk [G]$ is
injective as an $H$--module in order to guarantee the convergence of
the Lyndon-Hochschild-Serre spectral sequence --that relates the
cohomology of $G$, $H$ and $G/H$--.

The necessary injectivity result is a direct consequence of the
equivalence of (2) and (4) in Theorem \ref{theo:mainCPS} and it was
treated and proved in certain cases in the mentioned papers
\cite{kn:ho} and \cite{kn:inj}. For example, the injectivity of
$\Bbbk[G]$ as a rational $H$--comodule and the cohomological
consequences, were established in \cite[Prop.~2.2]{kn:inj} but only for
the case that the integrals are multiplicative --strong restriction
that rarely occurs except in the case of unipotent subgroups.  As far
as we are aware, the injectivity of $\Bbbk [G]$ as an $H$--module ,
for $H$ normal in $G$ was proved in full generality only much later in
\cite{kn:CPS}, \cite{kn:matsuhab} and \cite{kn:ob} (the three articles
appeared in 1977).  Non multiplicative general integrals appeared
around 1977, even though at first they were used in a subordinate way
to produce multiplicative ones.

Concerning this fact, we mention the following two results from
\cite{kn:CPS}.  In Proposition 1.10 (attributed to Hochschild:
\cite[Prop. 2.2]{kn:ho}) the author proves that if the closed
inclusion $H \subseteq G$ of affine algebraic groups admits an
\emph{equivariant cross--section}, then $\Bbbk[G]$ is injective as an
$H$--module. Such a cross section is a closed subvariety $S \subseteq
G$ such that the map given by multiplication $(s,x) \mapsto sx: S
\times H \to G$ is an isomorphism of varieties. The proof of the
injectivity result follows from the fact that the $H$--module algebras
$\Bbbk[S] \otimes \Bbbk[H]$ and $\Bbbk[G]$ are equivariantly
isomorphic with respect to the natural actions on each tensor factor
and endowing $\Bbbk[S]$ with the trivial $H$--action.

The use of integrals (without mentioning the name) appears in the
following theorem where the authors deal with the relationship between
the existence of a total integral with values in $\Bbbk [X]$ and the
existence of affine quotients of $X$ --at least for the case of a
unipotent group--. This situation can be generalized for non unipotent
groups, but one needs to restrict the variety $X$ to be an affine
algebraic group as in Theorem \ref{theo:mainCPS}.
 
\begin{theo}\cite[Thm.~3.1]{kn:CPS}
  Let $U$ be a connected unipotent group acting on an affine variety
  $X$. The following are equivalent.
  \begin{enumerate}
  \item $\Bbbk[X]$ is a rationally injective $U$--module.
  \item There is a $U$--equivariant morphism of varieties $\rho: X \to
    U$, (i.e., there is a $U$--equivariant algebra homorphism
    $\Bbbk[U] \to \Bbbk[X]$).
    \item There is a $U$--module homomorphism $\alpha:\Bbbk[U] \to
      \Bbbk[X]$ with $\alpha(1)=1$.
  \end{enumerate}
  When these conditions are satisfied, the quotient $X/U$ exists and
  is affine.
\end{theo}
\begin{proof}
  
  \noindent (1) $\Rightarrow$ (2) As $U$ is unipotent one can write
  $\Bbbk[U]$ as $\Bbbk[U]=\Bbbk[X_1,\cdots,X_n]$ with the property
  that if $P_i=\Bbbk[X_1,\cdots,X_i]$ then, for all $u \in U$, $u\cdot
  X_i \cong X_i (\!\!\!\mod P_{i-1})$. Then we start with $P_0=\Bbbk$
  for which we take the inclusion $\Bbbk \to \Bbbk[X]$ and construct
  by induction a $U$--equivariant algebra homomorphism $\alpha_i:P_i
  \to \Bbbk[X]$. Given $\alpha_{i-1}: P_{i-1} \to \Bbbk[X]$ we extend
  it as a $U$--equivariant morphism of $U$--modules $\beta_i:P_i \to
  \Bbbk[X]$ using the injectivity of $\Bbbk[X]$. Then, define
  $\alpha_i$ as the morphism of algebras that on the generators take
  values $\alpha_i(X_j)=\beta_i(X_j)$ for $1 \leq j \leq i$. It is
  easy to see that $\alpha_i$ is $U$--equivariant.

  \noindent (3) $\Rightarrow$ (1) This is the content of Theorem
  \ref{theo:intobsinj} item (2). See also the comment that follows
  after the proof.
  
It is clear that the quotient variety $X/U$ will be the cross--section
associated to $\rho$, i.e. $\rho^{-1}(1_U)$.
  \end{proof}


Nowadays, all these considerations have been proved to be valid in a
more general framework. In particular the
theory Hopf--Galois extensions is
well established --see for example \cite{kn:montgomery} for an
exposition of the original results of \cite{kn:sche}--. From today's
perspective one can say that \cite[Thm.~3.1]{kn:CPS} is a predecessor
of the theory that relates the existence of integrals with the Galois
theory of Hopf algebras as in \cite{kn:doita} --see
\cite{kn:montgomery} for a comprehensive exposition and a complete
bibliography--.

In a parallel development, Sweedler collected in his classical book
\cite{kn:Sbook} (1969) the basic properties of the (scalar) integrals
in the set up of general Hopf algebras. Therein he also proved, a
generalization for arbitrary Hopf algebras of Hochschild's result
guaranteeing that the existence of an (scalar) total integral for the
Hopf algebra of an affine algebraic group is equivalent to the
complete reducibility of the representations of the group
(\cite{kn:hobook0}). The general situation of the existence of total
$H$--integrals with values in $\Bbbk[G]$ for $H \subseteq G$ and its
relation with semisimplicity, appeared first in \cite{kn:firstwork}.

These developments culminate beautifully in a series of articles by
Y.~Doi and later by Y.~Doi and M.~Takeuchi starting in 1983. The
authors define the general notion of total integral from a Hopf
algebra $H$ in an $H$--comodule algebra $A$ and prove the
corresponding injectivity result as well as many other interesting
properties of the category of the $(A,H)$--comodules.  (see
\cite{kn:doi1}, \cite{kn:Doitot} and \cite{kn:doita}).
\subsection{Observability, exactness and quotients}

In this section we complete the proof of IMAQ's Theorem 4.3 showing the relation
of strong observability with the exactness of the induction functor
and also with the affineness of the associated homogeneous space.

We need first a proof of the fact that the exactness condition implies
the observability.

\begin{prop} Assume that $H \subseteq G$ is a closed inclusion of affine algebraic groups. If $H$ is exact in $G$ then, $H$ is observable in $G$.
\end{prop}
\begin{proof} Take $M \in \modG$ and consider the
  morphism $\pi \otimes id: \Bbbk[G] \otimes M \to \Bbbk[H] \otimes
  M$, that is clearly a morphism of $(\Bbbk[G],H)$--m\'odules (see
  Definition \ref{defi:exactfirst}) provided that we endow $\Bbbk[H]$
  with the structure of $\Bbbk[G]$ module given by $\pi$. The
  associated morphism obtained by restriction to the $H$--fixed part
  is the map ${}^H(\Bbbk[G] \otimes M) \to {}^H(\Bbbk[H] \otimes
  M)=M$, $\sum f_i \otimes m_i \mapsto \sum f_i(1)m_i$.  Thanks to the
  exactness hypothesis we deduce that this morphism $E_M$ --that is
  the counit of the adjunction between induction and restriction-- is
  surjective.  This is one of the possible characterizations of
  observability and hence the result is proved (see also
  \cite[Lemma 4.2]{kn:CPS} for another line of reasoning).
\end{proof}
The relation of observability and the induction functor is treated below in Section \ref{section:obsadj}: Definition \ref{def:induction} and Lemma \ref{lema:obsvarepsilon}.

We will need for the proof the following easy and handy Lemma that appears for example in \cite[Theorem 1.4.49]{kn:nosotros2}, and that guarantees that within the class of quasi--affine varieties, the validity of the Nullstellensatz characterizes the affine ones.
\begin{theo}\label{theo:convnull} Assume that $X$ is a quasi--affine variety with the property that if $J$ is an arbitrary proper ideal $J \subsetneq \mathcal O_X(X)$, then $\mathcal Z(J) \neq \emptyset$, then $X$ is affine. In particular if $H$ is an observable subgroup of $G$,
  if for all $J \subsetneq {}^H\Bbbk[G]$ we also have that $J\Bbbk[G] \neq \Bbbk[G]$, then $G/H$ is affine.  
  \end{theo}
\begin{theo}\label{theo:strgeom} Assume that $H \subseteq G$ is a closed inclusion of affine algebraic groups. The following three conditions are equivalent:
  \begin{enumerate}
  \item The subgroup $H$ is exact in $G$;
  \item The homogeneous space $G/H$ is affine;
    \item There is a total integral $\sigma: \Bbbk[H] \to \Bbbk[G]$. 
    \end{enumerate}
\end{theo}
\begin{proof} We prove that (2) $\Rightarrow$ (1) folowing Haboush's argument in \cite{kn:matsuhab}. For a rational $G$--module $M$ and $U \subseteq G/H$ open in $G/H$, we consider the usual diagonal action of $H$ on $\mathcal O_G(\pi^{-1}(U)) \otimes M$ and define $\mathcal I_M$, the sheaf on $G/H$ such that $\mathcal I_M(U)={}^H(\mathcal O_G(\pi^{-1}(U)) \otimes M)$.  It is clear that the global sections of this sheaf is the induced module ${}^H(\Bbbk[G] \otimes M)$ and a direct computation shows that  the stalk of the sheaf $\mathcal I_M$ at $eH \in G/H$ is $M$.  Hence, it is clear that for an exact sequence $0\to P \to Q \to R \to 0 \in \modH$, the sequence $0\to \mathcal I_P \to \mathcal I_Q \to \mathcal I_R \to 0$ is also exact. In the situation that $G/H$ is affine, Serre's cohomological characterization of affineness guarantees that the sequence of global sections of the above sheaf is also exact. This means that the induction functor is exact and it follows easily that this implies that $H$ is exact in $G$. 

  The proof that (1) $\Rightarrow$ (2) is as follows, from the
  exactness hypothesis we deduce that $G/H$ is quasi affine. In order
  to apply Theorem \ref{theo:convnull} take $J \subsetneq
  {}^H\Bbbk[G]$ a proper ideal. In the case that $J\Bbbk[G]=\Bbbk[G]$,
  we can find $\{j_1,\cdots,j_n\} \subseteq J$ such that the morphism
  of $(\Bbbk[G],H)$ modules $\Phi:\bigoplus_{i=1}^n\Bbbk[G] \to
  \bigoplus_{i=1}^n \Bbbk[G]$, $\Phi(g_1,\cdots,g_n)=\sum g_ij_i$ is
  surjective.  Then, the morphism $\Phi:\bigoplus_{i=1}^n{}^H\Bbbk[G]
  \to \bigoplus_{i=1}^n {}^H\Bbbk[G]$ is also surjective and that
  means that $J={}^H\Bbbk[G]$.

  Next we prove that (1) $\Rightarrow$ (3).

  Let  $\iota: M\hookrightarrow  N$ be an inclusion of finite dimensional
rational $H$--modules and  consider 
the  diagram in  ${}_H{\mathcal M}{}$  
\[
\xymatrix { M \ar[d]_{\phi} \ar@{^{(}->}[r]^\iota & N
  \ar@{.>}[dl]^(0.3){{\widehat{\phi}}} \\ \Bbbk [G] & }
\]

Consider $X={\operatorname{Hom}}_{\Bbbk }\bigl(M,\Bbbk [G]\bigr)$ and
$Y=\operatorname{Hom}_\Bbbk \bigl(N,\Bbbk [G]\bigr)$ endowed with the
standard rational $(\Bbbk [G],H)$--module structure.  The inclusion
$\iota$ induces a surjective morphism of $(\Bbbk [G],H))$--modules.
From the exactness of $H$ in $G$, we conclude that
$\iota^*\bigl({}^H{Y}{}\bigr)={}^H{X}{}$. Any element
$\widehat{\phi} \in {}^H{Y}{}$ mapped into $\phi\in
{}^H{X}$ is the extension of $\phi$ we are looking for.
For the case of infinite dimensional $H$--modules a Zorn's Lemma type of argument does the job to extend the morphism in the above diagram. We have thus proved that $\Bbbk[G]$ is injective as an $H$--module. And this implies condition (3).

The proof that (3) $\Rightarrow$ (1) goes as follows. Take an arbitrary
$(\Bbbk[G],H)$--module $M$ and consider the map $\mathcal R_M:M
\to M : \mathcal R_M(m)=\sum \sigma(S(m_1))m_0$. It is easy to show
that $\mathcal R_M(M)={}^HM$ and that for a morphism $f:M \to N$ of $(\Bbbk[G],H)$--modules, $f \circ \mathcal R_M=\mathcal R_n \circ f$. From the commutativity of the following diagram:
\[\xymatrix{M\ar[r]^{\mathcal R_M}\ar[d]_{f}&{}^HM \ar[d]^{f|_{{}^HM}}
    \\N\ar[r]^{\mathcal R_N}&{}^HN,}\] we deduce that if $f$ is
surjective, so is the restriction $f|_{{}^HM}$. Hence $H$ is exact in
  $G$.
\end{proof}
\subsection{Strong observability and reductivity}\label{subsection:soandred}

In IMAQ, for example in Corollary 4.5 or in Remark 4.4, the notion of
strong observabity (viewed as an injetivity condition) is studied for
a closed inclusion $H \subset G$ in the case that $G$ is
reductive. This sort of considerations are also present in the
mentioned work of Haboush where (using different methods), similar
results are proved. For example in \cite[Proposition
  3.2]{kn:haboushred}, the author proves that \emph{if $H \subseteq G$
  is a closed inclusion of affine algebraic groups with $G$ reductive,
  then $G/H$ is affine if and only if $H$ is reductive}\footnote{The difficult part is the conclusion of the reductivity of $H$ from the geometric hypothesis about the quotient space $G/H$.}. This
assertion is also known as \emph{Matsushima's criterion} and appeared
for the first time in \cite{kn:matsu}, and later proofs appeared in
work by Borel and Harish--Chandra, Bialynicki-Birula, Richardson,
Haboush, Cline Parshall and Scott (IMAQ), etc. The last three works, are valid in arbitrary characteristic and were published more or less simultaneously. In the introduction to Richardson's paper \cite{kn:rich} appears the following citation of a letter from Borel to the author (1977):    
\begin{center}
  \begin{minipage} {.8\textwidth}{\em
    ... The fact that $G/H$ affine implies that $H$ is reductive, has
      been know for almost 15 years, although not formally published. But this was only because of the difficulty to give references for some necessary foundational material on \'etale cohomology. In fact, using the Chevalley groups schemes over $\mathbb Z$ it can be seen that the \'etale cohomology$\mod \mathbb Z/\ell \mathbb Z \,\,,\,\,\ell \text{ prime } \neq
      \operatorname{char} \Bbbk$ of a reductive $\Bbbk$--group, is the same as the ordinary cohomology of the corresponding complex group. If one takes for granted the existence of a spectral sequence for the fibration of a group by a closed subgroup, then it is clear that the proof given in my joint paper with Harish--Chandra goes over verbatim for arbitrary characteristic, using \'etale cohomology. This was pointed out to me by Grothendieck (in 1961 as I remember it) as soon as I outlined this proof to him. I have always found mildly amusing that the so called 'algebraic proof'  of Bialynicki--Birula is restricted to characteristic zero, while the 'trascendental' one is not. The fact mentioned above about the cohomology of reductive groups is proved by M. Raynaud (Inv. Mat. 6 (1968)) but, apart from that, it seems difficult even now to give clear-cut references to the basic facts on \'etale cohomology needed here, so a more direct proof such as yours is still useful. 
  }\end{minipage}
\end{center}

Nowadays it is clear that the mentioned criterion admits for arbitrary
characteristic, proofs that are much more elementary than the one
suggested by Grothendieck using \'etale cohomology.  In
\cite{kn:firstwork,kn:secondwork} the authors propose a different
perspective that yields an easy proof for the above result and many
others. For that, one has to reinterpret the condition of the
exactness of $K$ in $H$ as an assertion on the linear reductivity of
the action of $K$ on $H$ --or on $\Bbbk[H]$. In this case if we look at
the trivial action of $H$ on $\Bbbk$ we obtain the concept of linear
reductivity. Using this viewpoint, Matsushima's criterion can be read
as follows: in the hypothesis that the action of $H$ on $\Bbbk$ is
linearly reductive we have that the action of $K$ on $H$ is linearly
reductive, if and only if the action of $K$ on $\Bbbk$ is linearly
reductive\footnote{In order to simplify the assertions we concentrate in this survey in the situation of linearly reductive actions (see Observation \ref{obse:genred})}.

\begin{defi}\label{defi:crucial2}\
\begin{enumerate}
\item Let $H$ be an affine algebraic group and $R$ a
rational $H$--module algebra. We say that the action of $H$
on $R$ is \emph{linearly reductive} if for every triple $(M,
J, \lambda)$ where $M\in (R,H)-\operatorname{mod}$, $J \subseteq R$ is an
$H$--stable ideal and $\lambda:M \rightarrow R/J$ is a
surjective morphism of $(R,H)$--modules; there exists an
element $m \in {}^HM$, such that $\lambda(m)=1+J \in R/J$.
In the context above, if the action of $H$ on $R$ is given,
we say that $(R,H)$ is a linearly reductive pair.
\item In the case that $R= \Bbbk [X]$ and the action of
$H$ on $R$ is linearly reductive we say that the action of
$H$ on $X$ is linearly reductive and also that the pair
$(H,X)$ is linearly reductive.
\end{enumerate}
\end{defi}
\begin{obse}\label{obse:genred}  A generalization of the notion of linearly reductive action to the concept of geometrically reductive action, can be defined (work in progress) and some of the considerations of the next theorem remain valid for this situation.
  \end{obse}
The proof of the theorem that follows is similar to others presented
before and we omit it (compare with the results in Section 4).
\begin{theo}\label{theo:characterizations} Let $H$ an affine
algebraic group and $R$ a rational $H$--module
algebra. Then, the following conditions are equivalent:
\begin{enumerate}
\item The action of $H$ on $R$ is linearly reductive.
\item If $\varphi:M \rightarrow N$ is a surjective morphism
in the category $\modRH$, then $\varphi\bigl({}^H\!M\bigr)={}^H\!N$.
\item There exists a total integral $\sigma:\Bbbk [H]
\rightarrow R$.
\item The $H$--module algebra $R$ is an injective object in
the category $\modH$.
\item Every object $M \in \modRH$ is injective in $\modH$.
\end{enumerate}
Morever, in the case that $H=U$ is unipotent, the action of $U$ on $R$
is linearly reductive, if and only if there is a multiplicative normal integral from $\Bbbk[U]$ into $R$.
\end{theo}
It is clear that the trivial action of $H$ on $\Bbbk$ is linearly reductive, if and only if $H$ is a lineraly reductive affine algebraic group. 

Once we free the notion of obervability of the restriction to the
group/subgroup situation, we acquiere a degree of flexibility that seems to provide a better understandig of the main issues of this
area. In that sense we mention below (without proofs) a few other
results from \cite{kn:secondwork}.

\begin{enumerate}
\item Let $K \subseteq H$ be a closed inclusion of affine algebraic groups. The following two conditions are equivalent: \begin{enumerate} \item The action of $K$ in $H$ and the action of $H$ in $H/K$ are linearly reductive \item $K$ is linearly reductive.\end{enumerate}
\item Let $K \subseteq H$ be as above and $R$ a rational $K$--module algebra and consider $R_H=\operatorname{Ind}_K^H(R)$ the induced $H$--module algebra. Assume moreover that the action of $K$ on $H$ is lineraly reductive. Then if the action of $H$ on $R_H$ is linearly reductive, so is the action of $K$ on $R$. For the definition of the functor $\operatorname{Ind}_K^H$ see Section \ref{section:obsadj}.
  \item ({\em Generalized Matsushima's criterion.}) Suppose that we have $K \subset H$ a pair given by a group and a subgroup, and that $R$ is an $H$--module algebra with the property that the action of $H$ on $R$ is linearly reductive. Then if the action of $K$ on $H$ is linearly reductive, then the action of $K$ on $R$ is linearly reductive. 
  \end{enumerate}
\section{Observable adjunctions}\label{section:obsadj}

The concept of \emph{observable adjunction} and of \emph{observable module category} appeared in 2006 (see \cite{kn:obscat}) as a direct product of the following observations based in the consideration of the monoidal categories $\modG$ and $\modH$ instead of the groups $G$ and $H$.

\medskip
Let $H \subseteq G$ be a closed inclusion of affine algebraic groups
and let $\mathcal D= {}_H\mathcal M$ and $\mathcal C= {}_G\mathcal M$
be the corresponding categories of rational representations. Call
${\mathbb L}:\mathcal C \rightarrow \mathcal D$ the restriction
functor, usually denoted as $\operatorname{Res}^H_{G}$, from rational
$G$--modules to $H$--modules.

It is well known that the monoidal functor ${\mathbb L}$ (see Definition \ref{defi:basictensor}) has a
right adjoint that is usually named as {the induction functor},
denoted as $\operatorname{Ind}_H^G$ and herein abbreviated as $\mathbb
R$.
\begin{defi} \label{def:induction} If $H \subseteq G$ is a closed inclusion of affine algebraic groups and $M \in {}_H\mathcal M$, we endow $\Bbbk[G] \otimes M$ with a structure of $H$--module acting on the left, and with a left structure of $G$--module where $x \in G$ acts as $x^{-1}$ on the right in the first tensor factor, and define $\mathbb R(M)$ as the $G$--module $\mathbb R(M):=\operatorname{Ind}_H^G(M):={}^H(\Bbbk[G] \otimes M)$. If $f:M \to M'$ is a morphism of rational $H$--modules, we define $\operatorname{Ind}_H^G(f):=(\operatorname{id} \otimes f)|_{{}^H(\Bbbk[G] \otimes M)}$. \end{defi}

It is well known (see for example \cite[Corollary
  7.7.12]{kn:nosotros2}) that $\mathbb L \dashv \mathbb R$
(i.e. $\mathbb L$ is the left adjoint of $\mathbb R$) or in explict
terms that: for all $M \in \modH$ and $N \in \modG$ there is a natural
isomorphism (in the category of $\Bbbk$--spaces)
$\operatorname{Hom}_H(\operatorname{Res}_G^H(N),M) \cong
\operatorname{Hom}_G(N,\operatorname{Ind}_H^G(M))$. In the classical
literature the above isomorphism was called the {\em Reciprocity law}.

The counit of the adjunction is the following family of maps:

\begin{equation}\label{eqn:counit}  \varepsilon_M:{}^H(\Bbbk[G] \otimes M)\to M  \quad,\quad \varepsilon_M(\sum f_i \otimes m_i)=\sum f_i(1)m_i \text{ for }\, \sum f_i \otimes m_i \in {}^H(\Bbbk[G] \otimes M).
\end{equation}

\medskip

The observability can be characterized in terms of the natural transformation $\varepsilon$.

\begin{lema}\label{lema:obsvarepsilon} In the above situation $H \subseteq G$ is observable if and only if for all $M \in \modH$, $\varepsilon_M:{}^H(\Bbbk[G] \otimes M)\to M $ is surjective.
\end{lema}
\proof

We prove that if for all $M \in {}_H \mathcal M$, the counit
$\varepsilon_M: {}^H(\Bbbk[G] \otimes M) \rightarrow M$ is
surjective, then $H$ is observable in $G$.

\medskip

We use the characterization in terms of extendible
characters.  Let $\chi$ a character of $H$, consider the character
$\chi^{-1}$ and write as $\Bbbk_{\chi^{-1}}$ the one dimensional
$H$--module defined by $\chi^{-1}$.

It is not hard to see that \[\operatorname{Ind}_H^G(\Bbbk_{\chi^{-1}})=
\{f \in \Bbbk[G]: x\cdot f= \chi(x)f, \forall x \in H\}= \Bbbk[G]_\chi\]   and that $\varepsilon:\Bbbk[G]_\chi
\rightarrow \Bbbk$ is the evaluation at the identity element of $G$.

  Using the surjectivity of $\varepsilon$  we can guarantee the existence of  $f\in \Bbbk[G]_\chi$ such that $f(1)=1$ and then $f$ is a non zero $\chi$--semi invariant.

 Next we show that if $H \subseteq G$ es observable then $\varepsilon$ is surjective for all $M \in \modG$.

First observe that if every $H$--representation $M$ can be extended to
a $G$--representation $N$, $M \subseteq N$, by dualization every
$H$--representation can be obtained as the projection of a
$G$--representation.  Hence it is clear that $H \subseteq G$ is
observable, if and only if for an arbitrary $H$--module $M$ there is a
$G$--module $N$ and a surjective morphism of $H$--modules such that $N
\twoheadrightarrow M$.
    \medskip

In this situation the universal property of the adjunction guarantees the existence of a map as in the diagram. 
\begin{center}
$\xymatrix{& \operatorname{Ind}_H^G(M) \ar[d]^{\varepsilon_M} \\ 
N \ar@{>>}[r] \ar@{-->}[ur]&M}$
\end{center}
 
\medskip
The surjectivity of the horizontal map implies the surjectivity of the vertical map $\varepsilon_M$.  
\qed
 
The above result is the justification for the following definition of \emph{obserevable action}.
First we introduce some nomenclature.

\begin{defi}\label{defi:basictensor}
   A monoidal category is a sextuple $\mathcal C= (\mathcal C,
   \otimes, \Bbbk, \Phi, \ell,r)$ where $\mathcal C$ is a category,
   $\otimes: \mathcal C \times \mathcal C \rightarrow \mathcal C$ is a
   functor, $\Bbbk$ is a fixed object, the unit; $\Phi$ is a natural
   isomorphism: the associativity constraint with components
   $\Phi_{c,d,e}: (c \otimes d) \otimes e \rightarrow c \otimes (d
     \otimes e)$, $\ell$ and $r$ are the unit constraints,
   that are natural isomorphisms with components $r_c: c \otimes
     \Bbbk \rightarrow c$ and $\ell_c: \Bbbk \otimes c \rightarrow
     c$.  Moreover, all these data satisfy certain coherence
   conditions --commutative diagrams (see MacLane's classic book: {\em Categories for the working mathematician}: \cite{kn:CWM}).

   If $\mathcal C$ and $\mathcal D$ monoidal categories and
   $T:\mathcal C\to\mathcal D$ is a functor a (strong) monoidal
   structure in $T$ is a natural isomorphism $T(c) \otimes T(d)
   \rightarrow T(c \otimes d)$ and an isomorphism $\Bbbk \rightarrow
   T(\Bbbk)$ with certain coherence conditions (see Joyal and
   Street:  \emph{Braided tensor categories.} \cite{kn:JS}).  A
   monoidal functor is a functor together with a monoidal structure.

   \smallskip
   \noindent
Given a monoidal category, a $\mathcal C$--module category is a category $\mathcal M$ together with a functor $\boxtimes: \mathcal C \times \mathcal M \to \mathcal M$ and natural isomorphisms \[\mu_{x,y,m}:(x \otimes y)\boxtimes m \to x \boxtimes (y \boxtimes m)\,,\, \lambda_m: \Bbbk \boxtimes m \to m,\] with compatibility conditions that we omit and involve the associativity constraint $\Phi$ and also the left and right unit constrains $\ell,r$. 
\end{defi}

From now on we assume that all categories are $\Bbbk$--linear and that the tensor structures and associated natural transformation are compatible with the linear structure.

\begin{defi} A non--trival module category over a tensor category $\mathcal C$ is said to be simple if any proper submodule category is trivial. The trivial module category is the category $\mathcal M=0$.  
  \end{defi}

\begin{defi}\label{defi:obsadjunction}
 Let $\mathcal C, \mathcal D$ be monoidal categories and  $\mathbb L: \mathcal C
\rightarrow \mathcal D$ a monoidal functor. Suppose that $\mathbb L$ admits
a right adjoint functor $\mathbb R: \mathcal D \rightarrow \mathcal C$.
 and call $\varepsilon_d: \mathbb L\mathbb Rd \Rightarrow d$ the counit. 
If $\varepsilon: \mathbb L\mathbb R \Rightarrow \id: \mathcal D \rightarrow \mathcal D$
  is a {surjective} natural transformation, we say that
   {$\mathcal D$ is observable in $\mathcal C$ and
    that the pair $(\mathbb L,\mathbb R)$ observes $\mathcal D$ in $\mathcal C$}.
\end{defi}
\begin{defi} In the above context  we endow $\mathcal D$ with a structure of $\mathcal C$ module category by the following rule: $\boxtimes:\mathcal C \times \mathcal D \to \mathcal D$ is $c \boxtimes d:=\mathbb L(c) \otimes_{\mathcal D} d$.
  \end{defi}

The following theorem illustrates the use of this concept in the theory of module categories.

\begin{theo}\cite[Theorem 2.3]{kn:obscat}
    Given an observable adjunction $\mathbb L \dashv \mathbb R$, $\mathbb L:\mathcal C \to \mathcal D\,,\, \mathbb R: \mathcal D \to \mathcal C$, if $\mathcal D$ is ind--rigid and the adjunction is observable  then $\mathcal D$ is simple as a $\mathcal C$--module category. 
\end{theo}

In the mentioned paper, the above considerations are used to study in some concrete cases the ideas related to the general definition of observability in particular, it is treated the case of Hopf algebra quotients $\pi:A \to B$ and the situation of the category of the linearized sheaves of a $G$--variety.

\section{Observable actions of groups on varieties}

\subsection{Brief description of the major results}

To illustrate the basic ideas of the current section we revisit some of the relevant results around the concept of observable subgroup $H$ of a connected group $G$. Consider the following four equivalent properties of a closed inclusion $H \subseteq G$.  

\begin{enumerate}
\item[(1)] For every $H$--stable and closed
subset $Y \subset G$ there is a non zero $H$--invariant polynomial function
that is zero on $Y$.

\item[(2)] The homogeneous space $G/H$ is a quasi--affine variety.

\item[(3)]  ${}^H{\bigl[ \Bbbk [G]\bigr]}{} =
\bigl[{}^H{\Bbbk [G]}{} \bigr]$.

\item[(4)]
 For every character $\rho \in \operatorname{\mathcal X}(H)$ there is
 a non zero polynomial $f \in \Bbbk [G]$, with the property that for
 all $x \in H$, $x\cdot f= \rho(x) f$, i.e.~every character is
 extendible.
\end{enumerate}

Around 2010 it was observed by Renner and Rittatore in the paper: {\em Observable actions of algebraic groups} (abbreviated as OAAG) (see \cite{kn:oaag}), that if (1) is taken as the definition of observable subgroup, it can be easily and profitably
generalized, by taking an arbitrary action of a group on a variety
rather than the action of a subgroup in a larger group.

Regarding this idea the following definition appeared in the mentioned paper:
\begin{defi}\label{defi:maindef}
Assume that $H$ is an affine algebraic group and that $X$ is an affine $H$--variety. The action of $H$ on $X$ is said to
be \emph{observable}, if every $H$--stable and closed subvariety $Y
\subset X$ admits an $H$--invariant polynomial function that is zero
on $Y$.
\end{defi}

In this more general situation, some adaptations are needed in order
to obtain results similar to the ones listed above. Here we just give
a succint description and more details appear later.

For example,
concerning the equivalence of conditions (1) and (3), in this general
case one needs to consider also the set $\Omega(X)=\{x \in X: O(x) \text{ is closed and of maximal dimension}\}$  in which case the
valid result guarantees
that the following two conditions (a) and (b) taken together, are
equivalent to the observability of the action: (a) $\bigl[{}^H\Bbbk
  [X]\bigr]={}^H{\bigl[\Bbbk [X]\bigr]}{}$; (b) $\Omega(X)$ has non--empty interior.

This general result is consistent with the case of group--subgroup,
because in the case that $H \subset G$, one has that $\Omega(G)=G$.

The characterization of observability in terms of the quasi--affineness
of the homogeneous space $G/H$, also has a version in the generalized
context guaranteeing the existence of a geometric quotient $X/H$ in a principal $H$--invariant open subset of $X$. 

For the above characterization of the observability of subgroups in
terms of the extension of characters, one has also some partial
results when generalizing: if the group $H$ acting on the affine
variety $X$ is solvable (or if the variety is factorial), the action
is observable if and only if the set of extendible characters is a
group (the concept of extendible character can be defined in exactly
the same manner as before).

\begin{defi}\label{defi:extenchar} If $H$ is an affine algebraic group acting regularly on the affine variety $X$. A character $\chi:H \to \Bbbk$ is said to be extendible, if there is a non zero polynomial $f \in \Bbbk[X]$ with the property that $x\cdot f=\chi(x)f$, for all $x \in H$.
\end{defi}

It is interesting to notice that there is a close
relation between the concepts of observable action and unipotency: 
indeed it can be shown that a group is \emph{universally observable}
(i.e.~its action is observable in any variety where it acts
rationally) if and only if it is unipotent.

The study --in the rather ``opposite'' direction-- of {\em observable actions of reductive groups} is also
interesting. For example, in OAAG it is shown that the action is
observable if and only if the set of closed orbits of maximal
dimension is not empty. Moreover, it can be proved that there is a
maximal $H$--stable closed subset of the original variety, such that
the restricted action is observable. In othere words, for reductive
groups all the actions are generically observable.

Even though, the study by the mentioned authors of this generalized
concept of observability has many other interesting results, in what
follows we limit ourselves in this short survey to the three areas of
results described above.

\subsection{A characterization of observable actions}

The result that follows is a first approximation to a geometric
perspective of the concept of observable actions.  Given a regular
action of an affine algebraic group $H$ on an affine variety $X$, if
the algebra of invariants ${}^H\Bbbk[X]$ is finitely generated we say
the the affinized quotient of $X$ by $H$ exists.  In that situation we
call $X/_{\operatorname{aff}\,}H$ the variety with the aforementioned
algebra of invariants as polynomial algebra and call $\pi: X \to
X/_{\operatorname{aff}\,}H$ the map associated to the natural
inclusion ${}^H\Bbbk[X] \subseteq \Bbbk[X]$.
\begin{theo}\label{theo:basicrel}
Assume that $H$ is an affine group acting regularly on an irreducible affine
variety $X$ and suppose that the affinized quotient $\pi: X 
\rightarrow X/_{\operatorname{aff}\,}H$ exists. If all the
fibers of $\pi$ are (closed) orbits, then the action is observable.
\end{theo}
\proof
If $Y\subset X$ is a $H$--stable closed subset with dense image in
$X/_{\text{aff}\,} H$, then $\pi(Y)$ contains an open subset of
$X/_{\text{aff}\,} H$. Hence, using our hypothesis concerning the
relationship between the fibers and the orbits, it follows that
$Y=\pi^{-1}\bigl(\pi(Y)\bigr)$, and as $\pi^{-1}\bigl(\pi(Y)\bigr)$ contains an open
subset of $X$ we conclude that $Y=X$.

It follows that if $Y\subsetneq X$ is an $H$--stable closed subset
strictly contained in $X$ it cannot have dense image; therefore there
exists $z\in (X/_{\text{aff}\,}H)\setminus \overline{\pi(Y)}$. Let
$f\in \Bbbk\bigl[ X/_{\text{aff}\,} H\bigr]={}^H\Bbbk[X]$ be such that
$f(z)=1$ and $f\bigl(\overline{\pi(Y)}\bigr)=0$. Then $f$ is a
non-zero invariant polynomial that is zero when restricted to $Y$.
\qed

 The theorem below
characterizes the observability in terms of conditions for the
invariant rational functions and a geometric condition on the orbits. The theorem just proved helps in the proof of one of the implications.

\begin{theo}
\label{theo:obsercharac}
Let $H$ be an affine group acting regularly on an irreducible affine
variety $X$. Then the following conditions are equivalent:

\noindent (1) The action of $H$ on $X$ is observable. 

\noindent (2) The following two conditions are satisfied:

\begin{enumerate}
\item[(a)] Every invariant rational function on $X$ is the quotient of two polynomials $\bigl[{}^H\Bbbk[X]\bigr]={}^H{\bigl[\Bbbk[X]\bigr]}{}$.

\item[(b)] The set $\Omega(X)$ has nonempty interior.
\end{enumerate}
\end{theo}

\proof We prove first that $(1) \Rightarrow (2)$. It follows from the
definition of observability that there is an invariant function $f \in
\Bbbk[X]$ with the property that $\emptyset\neq X_f\subset
X^{\text{max}}$, and then \cite[Theorem 7.3.5]{kn:nosotros2}
guarantees that $X_f \subseteq \Omega(X)$. This proves (b). Clearly
$\bigl[{}^H\Bbbk[X]\bigr]\subseteq {}^H{\bigl[\Bbbk[X]\bigr]}{}
$. Let $g\in {}^H{\bigl[\Bbbk[X]\bigr]}{}$, and consider the
ideal $I=\bigl\{ f\in \Bbbk[X] \mathrel{:} fg\in
\Bbbk[X]\bigr\}$. Then $I$ is $H$--invariant, and hence there exists
$0 \neq f\in {}^H\Bbbk[X]$ such that $fg\in {}^H\Bbbk[X]$, which
proves (a).

In order to prove the converse, i.e. (2) $\Rightarrow$ (1), take $f\in
{}^H\Bbbk[X]$ such that ${}^H\Bbbk[X_f]$ is finitely generated (the
existence of such an element $f$ is due to Grosshans in
\cite{kn:grossloc} and a proof appears also in \cite[Theorem
  7.5.6]{kn:nosotros2}).  It is not hard to see that the action of $H$
on $X$ is observable if and only if the action on $X_f$ is so. Thus,
we can assume without loss of generality that ${}^H\Bbbk[X]$ is
finitely generated. Let $\pi:X\to X/_{\text{aff}\,} H$ be the
affinized quotient, i.e. $X/_{\text{aff}\,} H$ is the affine variety
whose algebra of polynomial functions is ${}^H\Bbbk[X]$. By general
results on affinized quotients (e.g. \cite[Theorem
  14.7.1]{kn:nosotros2}) there exists $f\in {}^H\Bbbk[X]$ such that
$\pi^{-1}(y)=\overline{H\cdot x}$ for all $y\in V=(X/_{\text{aff}\,}
H)_f\cong X_f/_{\text{aff}\,} H$. Moreover, for a certain $(X_f)_0$,
an $H$--stable open subset of $X_f$, we have the following commutative
diagram:
\begin{center}
\mbox{
\xymatrix{
(X_f)_0\, \ar@{^(->}[r]\ar@{->>}[d]_-{\rho}&
X_f\ar@{->}[d]^-{\pi}\ar@{^(->}[r]&
X\ar@{->}[d]^-{\pi}\\
(X_f)_0/H\, \ar@{^(->}[r]_-\varphi& X_f/_{\text{aff}\,} H\,\ar@{^(->}[r]&
X/_{\text{aff}\,} H
}
}
\end{center}

\noindent where $\bigr(\rho\,,\,(X_f)_0/H\bigr)$ is a geometric
quotient. Since $\Bbbk\bigl((X_f)_0/H\bigr)=
{}^H{\Bbbk\bigl((X_f)_0\bigr)}{} = \Bbbk(X_f)^H$, it follows
by hypothesis that $\Bbbk\bigl((X_f)_0/H\bigr)
=\Bbbk(X_f/_{\text{aff}\,} H)$. Since $\rho$
and $\pi$ separate closed orbits, it follows that $\varphi$ is an open
immersion.

Since $\Omega(X)$ contains a nonempty open subset, it follows that $\Omega(X)\cap
(X_f)_0\neq \emptyset$. Let $g\in {}^H\Bbbk[X]$ be such that
$(X_f/_{\text{aff}\,} H)_g\subset (X_f)_0/H$. If $y\in
(X_f/_{\text{aff}\,} H)_g$, then $\pi^{-1}(y)=\overline{
  O(x)}$, where $x\in \Omega(X)\cap (X_f)_0$, hence $\pi^{-1}(y)$ is a closed
orbit of maximal dimension. Therefore, $\pi\big|_{_{X_{fg}}}: X_{fg}\to
(X_f/_{\text{aff}\,} H)_g\cong X_{fg}/_{\text{aff}\,} H$ is such that
all its fibers are closed orbits. Replacing $X$ by $X_{fg}$, we can
hence assume that all the fibers of the affinized quotient are closed
orbits.  Therefore, the proof of the observability of the action now
follows directly from Theorem \ref{theo:basicrel}.  \qed

\subsection{Observable actions and unipotency}

By the very definitions, both the unipotency of a group as well as the
observability of an action are conditions that can be formulated in
terms of the existence of enough invariants for certain actions of the
group in question. Therefore, it is natural to expect some close
connection between both concepts. This is illustrated below by showing 
that an affine algebraic group that is ``universally'' observable has
to be unipotent --- compare also with the notion of \emph{unipotent
  action} as defined in \cite{kn:firstwork} or \cite[Section
  7]{kn:secondwork}.

To implement the proof we use a result appearing in
\cite{kn:closedorb}, that guarantees that an affine algebraic group
$H$ is unipotent \emph{if and only if}  for all affine
$H$--variety $X$ the $H$--orbits on $X$ are \emph{closed}. 

\begin{theo}\label{theo:unipchar}
Let $H$ be an irreducible affine algebraic group such that every action
of $H$ on an affine algebraic variety is observable. Then $H$ is a
unipotent group.
\end{theo}

\proof
We first prove that every $H$--orbit on an affine $H$--variety $X$ is
  closed. Indeed, if $ O \subset X$ is an orbit, then the
  action of $H$ on the affine variety $\overline{O}$ is
  observable.  Hence, changing $X$ by $\overline{ O}$, we may
  assume that $X$ has an open (and dense) orbit $O$. If we
  call $I\subset \Bbbk[X]$ the $H$--stable ideal of $X \setminus
   O$, if this algebraic set is not empty, the ideal $I$ is
  not zero. If $f \in \Bbbk[X]$ is a $H$--fixed not zero function in
  $I$, it is clear that $f$ is constant on the orbit and hence on
  $X$. Thus, this constant function taking the value zero on a non
  empty set, has to be zero everywhere and this is a
  contradiction. Using the fact that we mentioned above, as all the orbits are closed we
  conclude that the group $H$ is unipotent.
\qed
  \subsection{Observable actions of reductive groups}
\label{subsection:obsred}

In this section, following \cite{kn:oaag}, we study the properties of
observable actions when the acting group is reductive. It can be
proved that given an action of $H$ on an affine variety $X$ there is a
maximal closed $H$--subvariety of $X$ such that the
restricted action is observable.
\begin{defi}
Recall that if $H$ is an affine group acting in the variety $X$, we
define the socle of $X$ --denoted as $X_{\operatorname{soc}}$ as:
\[X_{\operatorname{soc}}:=\bigcup_x\{O(x):\overline {O(x)}=O(x)\}.\] 
\end{defi}
\begin{theo}
\label{theo:socobser}
Let $H$ be reductive group acting on an
affine algebraic variety $X$.  Then the action is observable if and
only if $\Omega(X) \neq \emptyset$. In particular,
$X_{\operatorname{soc}}$ is the largest $H$--stable closed subset
$Z\subset X$ such that the restricted action $H\times Z\to Z$ is
observable.
\end{theo}

\proof If the action is observable, it follows from Theorem
\ref{theo:obsercharac} that $\Omega(X) \neq \emptyset$.  Assume now
that $\Omega(X) \neq \emptyset $ and let $Z\subsetneq X$ be a
$H$--stable closed subset and call $I$ the ideal associated to $Z$ ;
we want to show that ${}^HI \neq \{0\}$. If $\Omega(X) \subset Z$ it
follows that $Z=X$; hence $\Omega(X) \setminus Z\neq \emptyset$.
Recall that the semi--geometric quotient $\pi:X\to X /
H=\operatorname{Spm}\bigl({}^H\Bbbk[X]\bigr)$ separates closed orbits
--$\operatorname{Spm}$ is the maximal spectrum functor.  It follows
that $ \Omega(X) \setminus \pi^{-1}\bigl(\pi(Z)\bigr)\neq \emptyset$,
since the closed orbits belonging to $Z$ and
$\pi^{-1}\bigl(\pi(Z)\bigr)$ are the same. Let $O\subset \Omega(X)
\setminus Z$ be a closed orbit. Then $\pi^{-1}\bigl(\pi( O)\bigr)=O$,
again because $\pi$ separates closed orbits. Since $\pi$ also
separates $H$--stable closed subsets, it follows that there exists
$f\in {}^H\Bbbk[X]$ such that $f\in I'\subset I$ where $I'$ is the ideal of $\pi^{-1}\bigl(\pi(Z)\bigr)$ and $f(
O)=1$; in particular, $f\in {}^H I\setminus \{0\}$ and the
action is observable.

It follows by the very definition of $X_{\operatorname{soc}}$ that
$\Omega(X_{\operatorname{soc}}) \neq \emptyset$. Let $Z$ be an $H$--stable
irreducible closed subset such that the restricted action is
observable; then $\Omega(Z)$ is a nonempty open subset of $Z$,
consisting of closed orbits in $Z$, and hence in $X$. It follows that
$Z=\overline{\Omega(Z)}\subset X_{\operatorname{soc}}$.  If $Y$ is any
$H$--stable closed subset, it can be proved that
the restriction of the action to any irreducible component $Z$ is
observable, and hence $Y\subset X_{\operatorname{soc}}$.  \qed

\begin{theo}
  Let $H$ be a reductive group acting on an affine variety $X$ and call
  $I_0$--the ideal associated to $X_{\operatorname{soc}}$--. Then $I_0$
  is the largest $H$--stable ideal
such that ${}^HI=(0)$.
\end{theo}

\proof Let $I=\sum \{J: {}^HJ=(0)\}$ be the sum of all $H$--stable
ideals such that ${}^HJ=(0)$, and consider the canonical $H$--morphism
$\varphi: \bigoplus \{J: {}^HJ=(0)\} \to I$. Since $\varphi$ is
surjective, it follows from the reductivity of $H$ that for every
$f\in {}^HI$ there exist $n\geq 0$ and $h\in {}^H \bigoplus \{J:
{}^HJ=(0)\}=(0)$ such that $\varphi(h)=f^{p^n}$, where
$\operatorname{char}\Bbbk=p$, then as our algebras are free of
nilpotents, we deduce that ${}^HI=(0)$.

Let $O\subset X$ be a closed orbit, call $Z$ the set of zeros of
$I$ and assume that $O\cap Z =\emptyset$. Since ${}^H\Bbbk[X]$
separates $H$--stable closed subsets, if follows that there exists
$f\in {}^H\Bbbk[X]$ such that $f\big|_{_O}=1$ and $f\big|_Z=0$,
hence ${}^HI\neq (0)$ and we get a
contradiction. Therefore, $X_{\operatorname{soc}}\subset Z$.

Observe that if $f\in {}^H{\bigl(\sqrt{I}\bigr)}$  is such that $f^n\in I$, it follows
that for any $a\in H$, then  $a\cdot (f^n)=f^n\in I$, and hence $f=0$.
 Thus, ${}^H\bigl(\sqrt{I}\bigr)=(0)$ and by maximality  then $I=\sqrt{I}$.
By  Theorem \ref{theo:socobser}, if we prove that the action
$ H \times Z \to Z$ is observable ($Z$ is the set of zeros of $I$), then
$X_{\operatorname{soc}}=Z$. But $\Bbbk[Z]\cong
\Bbbk[X]/I$, and hence  the $H$--stable ideals of
$\Bbbk[Z]$ are of the form $J/I$, were
$J\subset \Bbbk[X]$ is an $H$--stable ideal containing $I$. Then if
$J/I\neq (0)$ it follows that $I\subsetneq J$ and hence, by maximality
of $I$,
${}^HJ\neq (0)$. Thus, ${}^H(J/I)\neq (0)$, since ${}^H\Bbbk[X]$ injects in
$\Bbbk[X]/I$.
\qed
\section{Final remarks}

Arising in the late 1950s and early 1960s from questions about the
existence of faithfull representations of Lie groups, the concept of
observability in his development along almost sixty years reached out
in a profitable interaction with most of the crucial themes of
--geometric and algebraic-- invariant theory. Today the original concept together with his generalizations, should be considered as an indispensable element in the toolkit of modern invariant theory.

\end{document}